\documentclass[reqno,10pt,a4paper]{article}
\usepackage{amsmath}
\usepackage{amsthm}
\usepackage{amsfonts}
\usepackage{amssymb}
\usepackage{bbm}
\usepackage{indentfirst}
\usepackage{marginnote}
\usepackage{xcolor}
\usepackage[left=3cm,top=3cm,right=3cm,bottom=3cm]{geometry}
\usepackage{color}
\usepackage[colorlinks=true, linkcolor=blue, citecolor=blue]{hyperref}

\newtheorem{thm}{Theorem}[section]
\newtheorem{lem}[thm]{Lemma}
\newtheorem{co}[thm]{Corollary}
\newtheorem{pr}[thm]{Proposition}
\newtheorem{re}[thm]{Remark}
\numberwithin{equation}{section}

\begin{document}
	\title{Central limit theorem and Berry-Esseen bounds for a branching random walk with immigration in a random environment}
	\author{Chunmao Huang\footnote{Corresponding author  \newline \indent \ \ Email addresses: cmhuang@hitwh.edu.cn (Chunmao Huang), renyukun@mail.preprego.com (Yukun Ren), lrz2000hitwh@163.com (Runze Li).}, \ \ Yukun Ren, \ \ Runze Li
	\\
	\small{\emph{Harbin Institute of Technology (Weihai), Department of Mathematics, 264209, Weihai, China}}}
	\date{}
	\maketitle
	
	\begin{abstract}
		We consider a branching random walk on $d$-dimensional real space  with immigration in a time-dependent random environment. Let $Z_n(\mathbf t)$ be the so-called partition function of the process, namely, the moment generating  function of the counting measure describing the dispersion of individuals at time $n$. For $\mathbf t$ fixed, the logarithm   $\log Z_n(\mathbf t)$ satisfies a central limit theorem. By studying the logarithmic moments of the intrinsic submartingale of the system and its convergence rates,  we establish the uniform and non-uniform Berry-Esseen bounds corresponding to the central limit theorem, and   discover the exact convergence rate in the central limit theorem. \\
		

	\end{abstract}

	\section{Introduction}\label{BS1}
Branching random walk is an important field of stochastic processes. There are many achievements in the literature: see e.g. \cite{B19771, B19772, B1997, L1997}. Unlike the processes having a fixed  distribution for all particles in the classical branching random walk, the reproduction distributions may change with the time or the space location in reality. As an important extension, branching random walks in random environments  attract  extensive attention and have many developments. The random environments in time or in space were considered, and a lot of valuable topics were studied, involving   limit theorems, convergence of martingales, asymptotic behaviours; see e.g. \cite{greven, BCG93, bk04, C2007, Y2008, HL2014, GL2014, WH17, HW2020, XK2021}.
In this paper, we are interested in the time-dependent random environment which   is a stochastic process varying from generation to generation. 
The model of branching random walk  in time-dependent random environment (BRWRE)  was first proposed by  Biggins and Kyprianou \cite{bk04}, and it  is also proposed by Kuhlbusch \cite{k04},  called the  weighted branching process  in random environment. In this model,  reproduction distributions differ  from generation to generation according to a random environment varying with the time, and the limit behaviours of the system will be affected by the environment; see e.g. \cite{bk04, k04, HL2014, GL2014, WH17, GL2018, ZH22} for   related studies. Recent years, branching processes with immigration have made a lot of achievements; see e.g. \cite{CLR12, wl17, DL20, LVZ21, LL21, hww22}. These research indicate that not only the environment, but also the immigration will influence the asymptotic properties of the branching system. In order to characterize the cross effect of environment and immigration on branching random walks, we consider the branching random walk with immigration in time-dependent random environment (BRWIRE). As a combination of BRWRE and branching process with immigration, this model extends largely the classical branching process and branching random walk, and has a wide application. However, due to the   complexity of the model, there are  few relevant results on this subject, which limits   many practical applications.

 In this paper, we  are interested in  BRWIRE and will focus on the central limit theorem associated to the partition function of the process.  When  all the  particles stay  at the origin, BRWIRE degenerates to a branching process with immigration in random environment (BPIRE), and the  partition function  at time $n$ denotes  the population size of the $n$-th generation. For BPIRE, Wang and Liu  established a central limit theorem associated to  the population size in  \cite{wl17}, and then  showed the corresponding Berry-Esseen bound in \cite{WL21}. Meanwhile, for a branching process without immigration in random environment (BPRE), 
 the central limit theorem was established by Huang and Liu \cite{hl12}, the Berry-Esseen bound was shown  by Grama \emph{et al.}  \cite{glm17}, and 
 Gao \cite{G21} discovered the exact convergence rate in the central limit theorem. With our in-depth study for the case with immigration, we find that the moment conditions of these results can be completed or improved, and  it is also possible to obtain more profound conclusions. These studies and findings urge us to further explore the central limit theorem for BRWIRE. The purpose of the paper is to study the  central limit theorem associated to the partition function for BRWIRE, by  discovering  its speed of convergence in the form of the   (uniform and non-uniform) Berry-Esseen bounds as well as its exact convergence  rate. We want to extend the existing results on BPRE or BPIRE to BRWIRE while improving the conclusions and the moment conditions. In terms of research methods,    we shall adopt  similar approaches to  BPIRE, since BRWIRE and BPIRE have similar structure. The basic idea is that the logarithm of the partition function will satisfy similar limit theorems with similar convergence rates to the associated random walk  if  the intrinsic submartingale of the process converges  as fast as possible. However, 
 as BRWIRE contains the information of particle locations, many techniques applicable to BPIRE are no longer available for BRWIRE. So we also need to invent  new solutions based on the basic idea, especially in the study of convergence rates of the  intrinsic submartingale.

 \section{Model and main results}\label{BS11}
	
Let us describe the model in detail.  Here we consider \emph{a branching random walk with immigration in a time-dependent random environment (BRWIRE)}, which is an extension of the so-called branching random walk in a time-dependent random environment (BRWRE) by introducing the immigration.
The random environment, denoted by $\xi=(\xi_n)$,  is an independent identically distributed (i.i.d.) sequence of random variables indexed by the  time $n\in\mathbb N=\{0,1,2,\cdots\}$. Each realization of $\xi_n$ corresponds to two distributions $\eta_n=\eta(\xi_n)$ and $\iota_{n}=\iota(\xi_n)$ on $\mathbb{N}\times(\mathbb{R}^d)^{\mathbb N}$: $\eta_n$ determines the branching distribution  and $\iota_n$ determines the immigration distribution at time $n$.

When the environment $\xi=(\xi_n)$ is given, the branching random walk is performed  on the $d$-dimensional real space $\mathbb R^d$. At time $0$, there is an initial particle $\emptyset$ of Generation $0$ located at $\mathbf S_{\emptyset}=\mathbf 0\in\mathbb{R}^d$.
At time $1$, this particle is replaced by  $N=N({\emptyset})$ new born particles, located at $\mathbf L_i=\mathbf L_i({\emptyset })$, $i=1,2,\cdots,  N$, where the random vector
$\mathcal  X({\emptyset})=(N, \mathbf L_1, \mathbf L_{2},\cdots)$ is of distribution $\eta_0$.  At the same time, $V_0$ new immigrant particles  join the family,  with locations $\mathbf S_{0_0i}$, $i=1,2,\cdots, V_0$, where the random vector $\mathcal Y_0=(V_0, \mathbf S_{0_01}, \mathbf S_{0_02},\cdots)$ is of distribution $\iota_0$. All the new born particles and the new immigrant particles form   Generation $1$  of the system. Then, each particle of Generation $1$ produces at time $2$  a random number of offspring and some new  immigrants   come and join the family  at the same time, which form  Generation $2$ of the system. In general, each particle $u$ of Generation $n$ located at $\mathbf S_u$ is replaced at time $n+1$ by $N(u)$ new born particles $ui$ of generation $n+1$, $i=1,2,\cdots, N(u)$,  where $ui$ denotes the $i$-th child of $u$; the location of $ui$ satisfies
$$\mathbf S_{ui}=\mathbf S_u+\mathbf L_i(u),\qquad   i=1,2,\cdots,  N(u), $$
where the  random vector $\mathcal X(u)=(N(u),\mathbf L_1(u), \mathbf L_2(u),\cdots)$ is of distribution $\eta_n$; at the same time, $V_n$ new  immigrants   join the family, located at $\mathbf S_{0_ni}$, $i=1,2,\cdots,  V_n$, where the random vector $\mathcal Y_n=(V_n, \mathbf S_{0_n1}, \mathbf S_{0_n2},\cdots)$ is of distribution $\iota_n$. All the new born particles $ui$, $i=1,2\cdots, N(u)$, and all the new immigrant particles $0_n i$, $i=1,2,\cdots, V_n$, form Generation $n+1$ of the family.
Conditioned on the environment $\xi$, 
the random vectors $\mathcal  X(u)$ indexed by all the sequences $u$ and $\mathcal  Y_n$ indexed by $n$  are  conditionally independent of each other. 

Let $\mathbb{U}_X=\cup_{n\geq 1}\mathbb{N^*}^n\cup\{\emptyset\}$, $\mathbb U_Y=\cup_{n\geq 0}\{0_nu: u\in\cup_{n\geq 1}\mathbb{N^*}^n\}$ and $\mathbb{U}=\mathbb{U}_X\cup\mathbb U_Y$, where $\mathbb N^*=\{1,2,\cdots\}$.  For $u=u_1\cdots u_n\in\mathbb{U}$, we write $|u|$ for the length of $u$, with the convention that   $|\emptyset|=0$ and $|0_n|=n$. Let $\mathbb{T}$ be the family tree  with defining elements $\{N(u)\}_{u\in\mathbb{U}}$ and
$\mathbb{T}_n=\{u\in\mathbb{T}: |u|=n\}$ be the set of particles of Generation $n$ of the family. For $n\in\mathbb N$, define
$$Z_n(\mathbf t)=\sum_{u\in\mathbb{T}_n}e^{\mathbf t\mathbf S_u}\qquad (\mathbf t\in\mathbb R^d).$$
Here  and throughout the paper the product $\mathbf x \mathbf y$ denotes the inner product of $\mathbf x$ and $\mathbf y\in \mathbb R^d$.
Called by Physicians the \emph{partition function}, $Z_n(\mathbf t)$ is the moment generating function of  the measure which describes the dispersion of particles in Generation $n$.   In particular, when $\mathbf t=\mathbf 0$, $Z_n(\mathbf 0)$ is the population size of Generation $n$, and it forms a branching process with immigration in a random environment (BPIRE); moreover, if $V_n=0$ for all $n$, there is no immigration and $Z_n(\mathbf 0)$ forms a  branching process (without immigration) in a random environment (BPRE).  In this paper,  we  are interested in the central limit associated to $\log Z_n(\mathbf t)$,  as well as  its convergence rates, for $\mathbf t\in\mathbb R^d$ fixed.

Denote by $\mathbb{P}_\xi$ the so-called quenched law, i.e. the conditional
probability  when the environment $\xi$ is given. The total probability can be expressed as $\mathbb P(\mathrm{d}x, \mathrm{d}\xi)=\mathbb P_\xi(\mathrm{d}x)\tau(\mathrm{d}\xi)$, with $\tau$ the distribution of the environment $\xi$; it is usually called the annealed law. We also use $\mathbb P_{\xi,Y}$ to denote the  conditional probability  when both the environment $\xi$ and the immigrant sequence $Y=(\mathcal Y_n)$ are given. The expectations with respect to $\mathbb{P}_{\xi,Y}$, $\mathbb{P}_\xi$ and $\mathbb P $  will be denoted respectively by $\mathbb E_{\xi,Y}$, $\mathbb E_\xi $ and $\mathbb E$.

For $n\in\mathbb N$ and $\mathbf t\in\mathbb R^d$, set
$$Y_n(\mathbf t)=\sum_{i=1}^{V_n}e^{\mathbf t\mathbf S_{0_ni}}, \qquad
m_n(\mathbf t)=\mathbb E_\xi\sum_{i=1}^{N(u)}e^{\mathbf t\mathbf L_i(u)}\;\;\;\;(u\in\mathbb {N^*}^n),$$
$$\Pi_0(\mathbf t)=1  \qquad \text{and}\qquad \Pi_n(\mathbf t)=\prod_{i=0}^{n-1}m_i(\mathbf t)\quad (n\geq 1).
$$
It is known that the normalization 
\begin{equation*}
	W_n(\mathbf t)=\frac{ Z_n(\mathbf t)}{\Pi_n(\mathbf t)} 
\end{equation*}
is the intrinsic non-negative sub-martingale of the system. Let $\mathbb{T}^X$ be the family tree  with defining elements $\{N(u)\}_{u\in\mathbb{U}_X}$ and
$\mathbb{T}^X_n=\{u\in\mathbb{T}^X: |u|=n\}$ be the set of particles of Generation $n$ originating from the initial particle $\emptyset$. For $n\in\mathbb N$ and $\mathbf t\in\mathbb R^d$, set
\begin{equation*}
	\bar Z_n(\mathbf t)=\sum_{u\in\mathbb T^X_n}e^{\mathbf t\mathbf S_u}\qquad\text{and}\qquad \bar W_n(\mathbf t)=\frac{\bar Z_n(\mathbf t)}{\Pi_n(\mathbf t)}.
\end{equation*}
Then $\bar W_n(\mathbf t)$ is the so-called Biggins  martingale in  branching random walks, and hence it converges almost surely  (a.s.) to some limit $\bar W(\mathbf t)$ with $\mathbb E_\xi\bar W(\mathbf t)\leq 1$.

We will focus on the central limit theorem on $\log Z_n(\mathbf t)$ for $\mathbf t\in\mathbb R^d$ fixed. For the sake of brevity, we will omit $\mathbf t$  in some notation in the following description, such as writing  $Z_n=Z_n(\mathbf t)$, $W_n=W_n(\mathbf t)$ and $Y_n(\mathbf t)=Y_n$.
For $\mathbf t\in\mathbb R^d$, denote
$$\mu=\mu(\mathbf t)=\mathbb E\log m_0(\mathbf t)\quad\text{and}
\quad \sigma=\sigma(\mathbf t)=\sqrt{Var(\log m_0(\mathbf t))}$$
if the expectations exist. Throughout the paper, we assume that 
$$\mathbb P(N=0)=1,$$
which means that each particle produces at least one offspring. If $\mu>0$, it is known that  the submartingale $W_n$ converges a.s.  to  a non-degenerate limit under appropriate moment conditions.

\begin{lem}[\cite{hwl22+, WH19}, Convergence of $W_n$]\label{a.s.}
If $\mu>0$,  
$\mathbb E\log^+\frac{Y_0}{m_0(\mathbf t)}<\infty$, $\mathbb E\log \frac{m_0(p\mathbf t)}{m_0(\mathbf t)^p}<0$ and $\mathbb E\log \mathbb E_\xi \bar W_1^p<\infty$ for some $p>1$, then
 the limit
\begin{equation*}
W=\lim_{n\rightarrow\infty}W_n
\end{equation*}
exists a.s. with values in $(0,\infty)$.
\end{lem}

When the limit $W\in(0,\infty)$ exists a.s.,  we can establish a central limit theorem on $\log Z_n$. Set 
$$S_n=\log \Pi_n(\mathbf t)-n \mu=\sum_{i=0}^{n-1}(\log m_i(\mathbf t)-\mu).$$
It is clear that $S_n$ is a centred random walk with i.i.d. displacements $\{\log m_n(\mathbf t)-\mu\}$. Hence it satisfies the classical central limit theorem for  i.i.d. sequence, which says that as $n$ tends to infinity, 
$$\frac{S_n}{\sqrt{n}\sigma}\longrightarrow U\in\mathcal N(0,1)\quad \text{in distribution,}$$
if $\sigma\in(0,\infty)$.
Notice that 
$$\frac{\log Z_n-n\mu}{\sqrt{n}\sigma}=\frac{S_n}{\sqrt{n}\sigma}+\frac{\log W_n}{\sqrt{n}\sigma}.$$
If $W_n$ tends to $W\in(0,\infty)$ a.s., we have $\frac{\log W_n}{\sqrt{n}\sigma}\to 0$ a.s., so that
$$\frac{\log Z_n-n\mu}{\sqrt{n}\sigma}\longrightarrow U\in\mathcal N(0,1)\quad \text{in distribution.}$$
We state this central limit theorem as follows.	
	
		\begin{thm}[Central limit theorem]\label{tt1}
		If $\mu, \sigma\in(0,\infty)$,  
$\mathbb E\log^+\frac{Y_0}{m_0(\mathbf t)}<\infty$, $\mathbb E\log \frac{m_0(p\mathbf t)}{m_0(\mathbf t)^p}<0$ and $\mathbb E\log \mathbb E_\xi \bar W_1^p<\infty$ for some $p>1$, then 
		\begin{equation*}\label{te1}
			\lim_{n \to \infty}\sup_{x\in \mathbb R}\left | \mathbb{P}\left ( \frac{ \log{Z _{n}  }  -n\mu  }{\sqrt{n} \sigma  }\le x \right ) -\Phi \left ( x \right )  \right |=0,
		\end{equation*}
		where $\Phi(x)=\frac{1}{\sqrt 2\pi}\int_{-\infty}^x e^{-s^2/2}\mathrm{d}s$ is the standard normal distribution function. 
	\end{thm}	

We are interested in the convergence rates in this central limit theorem. In order to show these rates, 
 we introduce the following assumption:\\*\\ 
 $\mathbf{(H)}$ There exist constants $\varepsilon>0$ and $p>1$ such that
$$\mathbb E\left(\frac{m_0(p\mathbf t)}{m_0(\mathbf t)^p}\right)^\varepsilon<1\quad\text{and}\quad \mathbb E\left(\mathbb E_\xi\bar W_1^p\right)^\varepsilon<\infty.$$
This moment assumption  is used to ensure the existence of  the moments or logarithmic moments of $  W_n$ such that   $W_n$ converges with a fast rate, if $\varepsilon$ and $p$ satisfy certain restriction. In particular, for branching processes, namely, when $\mathbf t=\mathbf 0$, we mention that we just need (H) for $\varepsilon>0$ and $p>1$ without additional restrictions; about the existence of  moments  for BPRE or BPIRE, see \cite{huang14, glm17, HZG} for   positive moments, harmonic moments and logarithmic moments of $W_n$ respectively.

In Theorem \ref{tt1}, the condition $\sigma<\infty$ implies that the moment condition on $\log m_0(\mathbf t)$ is  $\mathbb E(\log m_0(\mathbf t))^a<\infty$ for $a=2$  at least. If a moment with   higher order  exists, such as $\mathbb E(\log m_0(\mathbf t))^a<\infty$ for  some $a>2$, it is known that the central limit theorem on $S_n$ will converge  very fast, with a polynomial rate $n^{-\delta/2}$ for  $\delta=\min\{a-2,1\}$. Such result is called the Berry-Esseen bound or Berry-Esseen inequality. When  (H) holds for certain  $\varepsilon$ and $p$, one may expect that $W_n$ converges very fast such that a similar result also holds for $\log Z_n$.

\begin{thm}[Uniform Berry-Esseen bound]\label{tt2}
Assume that $\sigma>0$, $\mathbb{E}(\log m_0(\mathbf t))^{a} < \infty$  for some $a>2$, $\mathbb E m_0(\mathbf t)^{-\alpha}<1$ and  $\mathbb E \left (\frac{Y_0}{m_0(\mathbf t)}\right)^\alpha<\infty$ for some $\alpha>0$.   Suppose that either of the following two conditions holds:
\begin{itemize}
\item[$\mathrm(i)$]$\mathbb E(\frac{e^{\mathbf t\mathbf L_1}}{m_0(\mathbf t)})^{-\alpha}<\infty$ and  $ {(H)}$ is true for some $\varepsilon >0$ and $p>\max\{1+\frac{1}{\varepsilon},\frac{2}{\varepsilon}\}$;
\item[$\mathrm(ii)$]$\mathbb E \left | \mathbf t \mathbf L_1 \right | ^a<\infty$ and $(H)$ is true for some $\varepsilon>0$, $p > \max\{1+\frac{a}{(a-r)\varepsilon},\frac{2a}{(a-r)\varepsilon}\}$ and $r\in(0,a)$.
\end{itemize}
Set  $\delta=\min\{a-2,1\}$ if (i) is satisfied, and $\delta=\min\{a-2,r,1\}$ if  (ii) is   satisfied.
Then
\begin{equation}\label{bs}
\sup_{x\in \mathbb R}\left | \mathbb{P}\left ( \frac{ \log_{}{Z _{n}  }  -n\mu   }{\sqrt{n} \sigma  }\le x \right ) -\Phi \left ( x \right )  \right |\le Cn^{-{\delta }/{2} },
	\end{equation}	
	where   $C>0$ is a  constant   depending on $\mathbf t$.
\end{thm}

Theorem \ref{tt2} gives the uniform Berry-Essen bound  corresponding to the central limit theorem on $\log Z_n$. The uniformity means that this bound  holds uniformly for all $x\in\mathbb R$ and it is independent of $x$.  There is some overlap between  the conditions (i) and (ii) of Theorem \ref{tt2}, but there is no inclusion relationship. However, in   view of the requirement for (H), the condition (ii) is stronger than the condition (i). For $\mathbf t=0$, under the common assumption (H) which now becomes $\mathbb E\left(\mathbb E_\xi\bar W_1^p\right)^\varepsilon<\infty$ for some $\varepsilon>0$ and $p>1$ without other restrictions,    the moment requirement $\mathbb E m_0(\mathbf 0)^\alpha<\infty$ in the condition (i) implies that $\mathbb E (\log m_0(\mathbf 0))^a<\infty$ for all $a>0$, which means that the condition (i) is stronger than the  condition  (ii); see \cite{WL21, HZG} for the Berry-Esseen bound for BPIRE.

In Berry-Esseen bounds, if the upper bound is a sequence of functions $f_n(x)$ depending on $x\in\mathbb R$, we call such result the non-uniform Berry-Esseen bound. For the random walk $S_n$, the non-uniform Berry-Esseen bound holds (see Lemma \ref{s4l1}, cf. \cite{L2, B66, P95}):  under the moment condition that  $\mathbb E(\log m_0(\mathbf t))^a<\infty$ for  some $a>2$,  if $\sigma>0$, then
\begin{equation}\label{nbss}
\left | \mathbb{P}\left ( \frac{  S_n }{\sqrt{n} \sigma  }\le x \right ) -\Phi \left ( x \right )  \right |\le Cn^{-{\delta }/{2} }(1+|x|)^{-a}\qquad (\forall \;x\in\mathbb R),
\end{equation}
where $\delta=\min\{a-2, 1\}$. Obviously, the non-uniform Berry-Esseen bound is weaker than the uniform one. Just like $S_n$,  we can also establish a non-uniform Berry-Esseen bound on $\log Z_n$.

\begin{thm}[Non-uniform Berry-Esseen bound]\label{tt20}
Assume that $\sigma>0$, $\mathbb{E}(\log m_0(\mathbf t))^{a} < \infty$  for some $a>2$, $\mathbb E m_0(\mathbf t)^{-\alpha}<1$ and  $\mathbb E \left (\frac{Y_0}{m_0(\mathbf t)}\right)^\alpha<\infty$ for some $\alpha>0$.   Set $\delta=\min\{a-2,1\}$. 
\begin{itemize}
\item[$\mathrm(a)$]If $\mathbb E(\frac{e^{\mathbf t\mathbf L_1}}{m_0(\mathbf t)})^{-\alpha}<\infty$ and  $ {(H)}$ is true for $\varepsilon >0$  and $p>\max\{1+\frac{1}{\varepsilon},\frac{2}{\varepsilon}\}$, then for   $\lambda\in(0, a)$, we have 
\begin{equation}\label{nbs}
\left | \mathbb{P}\left ( \frac{ \log_{}{Z _{n}  }  -n\mu   }{\sqrt{n} \sigma  }\le x \right ) -\Phi \left ( x \right )  \right |\le Cn^{-{\delta }/{2} }(1+|x|)^{-\lambda}\qquad (\forall \;x\in\mathbb R).
	\end{equation}
\item[$\mathrm(b)$]	If	$\mathbb E \left | \mathbf t \mathbf L_1 \right | ^a<\infty$ and $(H)$ is true for  some $\varepsilon>0$ and $p > \max\{1+\frac{a}{(a-\max\{2\delta,1\})\varepsilon},\frac{2a}{(a-\max\{2\delta,1\})\varepsilon}\}$, 
then there exists   $\lambda_0\in(0, a]$ such that \eqref{nbs} holds for $\lambda\in(0, \lambda_0)$.
\end{itemize}
\end{thm}

\begin{re}In Theorem \ref{tt20}(b), the value of $\lambda_0$ can be calculated precisely. Denote $$a^*=\sup\{a:\; \mathbb E \left | \mathbf t \mathbf L_1 \right | ^a<\infty, \;\mathbb{E}(\log m_0(\mathbf t))^a <\infty \}\quad \text{and}\quad  r^*=\min\left\{\frac{p\varepsilon-2}{p\varepsilon},\;\frac{(p-1)\varepsilon-1}{(p-1)\varepsilon}\right\}a^*.$$
It can be observed that
$$\max\{2\delta,1\}\le r^*\le a^*$$
and 
$$a^*\geq \max\left\{a, \;\;\max\left\{\frac{p\varepsilon}{p\varepsilon-2},\;\frac{(p-1)\varepsilon}{(p-1)\varepsilon-1}\right\}\max\{2\delta,1\}\right\}.$$
Set $q^*=q(r^*)$, where the function $q(r)$ is defined as follows:
if $\delta=1$, $q(r)=\frac{r}{2}$; if $\delta\in(0,1/2]$, $q(r)=r$; if $\delta\in(1/2,1)$,
\begin{align*}
q(r)=
\begin{cases}
r, & \text{if \;$r\geq\frac{\delta}{1-\delta}$,}\\
\frac{(1+\delta)-\sqrt{(1+\delta)^2-4r(1-\delta)}}{2(1-\delta)},& \text{if\; $2\delta\leq r<\frac{\delta}{1-\delta}$.}
\end{cases}
\end{align*}
Then $\lambda_0=\min\left\{\eta^*,\;\frac{\eta^*}{\eta^*+1}a^*\right\}$, where $\eta^*=\min\{r^*-1,\; q^*\}$. One can calculate that 
\begin{align*}
\lambda_0=\eta^*=\begin{cases}
\frac{r^*}{2},& \text{if $\delta=1$,}\\
r^*-1, & \text{if $\delta\in(0,1/2]$.}
\end{cases}
\end{align*}
When $\delta\in(1/2,1)$,  the situation is a little complicated: if $\delta\leq 2\sqrt{2}-2$, then we still have $\eta^*=r^*-1$; however, if $\delta> 2\sqrt{2}-2$, 
there exists two constants $r_1<r_2$ depending on $\delta$ such that   
\begin{align*}
\eta^*=\begin{cases}
r^*-1,& \text{if $r^*\leq r_1$ or $r^*\geq r_2$,}\\
q^*, & \text{if $r_1<r^*<r_2$,}
\end{cases}
\end{align*}
where $r_{1}=\max\left\{2\delta, \; 1+\frac{\delta-\sqrt{\delta^2-4(1-\delta)}}{2(1-\delta)}\right\}$ and $r_{2}=\min\left\{\frac{\delta}{1-\delta},\;1+\frac{\delta+\sqrt{\delta^2-4(1-\delta)}}{2(1-\delta)}\right\}$. 
From the representation of $\lambda_0$, we can see that the value of $\lambda_0$ may be very small and even near $0$  when $a^*$  and $\min\left\{\frac{p\varepsilon-2}{p\varepsilon},\;\frac{(p-1)\varepsilon-1}{(p-1)\varepsilon}\right\}$ are small.
In particular, if $a^*=\infty$, we  have $\eta^*=q^*=r^*=\infty$, so that $\lambda_0=\infty$.
\end{re}

Comparing Theorem \ref{tt20} with Theorem \ref{tt2}, we see that the condition (a) of Theorem \ref{tt20}  is the same as the  condition (i) of Theorem \ref{tt2},  but \eqref{nbs} can leads to \eqref{bs}, which means that the result of  Theorem \ref{tt20}  is weaker than that of Theorem \ref{tt2}. However, the condition (b) of Theorem \ref{tt20}  is a little stronger than the  condition (ii) of Theorem \ref{tt2}, which indicates that the non-uniform Berry-Esseen bound may not exist even if the uniform Berry-Esseen bound   exists. Comparing   \eqref{nbs} with \eqref{nbss}, we find that (by noticing $\lambda<a$) the non-uniform Berry-Esseen bound  on $\log Z_n$ is larger than the corresponding one on $S_n$. Due to the difference between $\log Z_n$ and $S_n$, it is difficult for $\lambda$ to reach its supreme $a$ through our approach, although  it can be infinitely close to $a$ under ideal conditions.

Using the  method  in the proof of Theorem \ref{tt20}, and combining it with the moments and convergence rates of $\log W_n$ for BPIRE (cf. \cite{HZG}), we can establish the non-uniform Berry-Esseen bound on $\log Z_n$ for BPIRE.

\begin{co}\label{cobs}
Fix $\mathbf t=\mathbf 0$. Assume that $\sigma>0$, $\mathbb{E}(\log m_0(\mathbf 0))^{a} < \infty$  for some $a>2$, $\mathbb E \left (\frac{Y_0}{m_0(\mathbf 0)}\right)^\varepsilon<\infty$ and  $\mathbb E\left(\mathbb E_\xi\bar W_1^p\right)^\varepsilon<\infty$ for some $\varepsilon>0$ and $p>1$. Set   $$a^*=\sup\{a:\; \mathbb{E}(\log m_0(\mathbf 0))^{a} < \infty\} \quad\text{and}\quad \lambda_0=a^*-1.$$ Then \eqref{nbs} holds for $\lambda\in(0, \lambda_0)$ and $\delta=\min\{a-2,1\}$. In particular, if $a^*=\infty$, then $\lambda_0=\infty$ and $\delta=1$. In addition, if $a^*<\infty$ and $\lambda\in[\lambda_0, a^*)$, then \eqref{nbs} also holds for this $\lambda$ and $0<\delta<\min\{a^*-2, a^*-\lambda\}$.
\end{co}

Compared with Theorem \ref{tt20},  the $\lambda_0$ in Corollary \ref{cobs} is at least more than $1$, and hence the phenomenon that $\lambda_0$ is near 0 will never happen, since there is no additional relationship on  $\varepsilon>0$ and $p>1$ in (H) for BPIRE.

If $\mathbb E (\log m_0(\mathbf t))^a<\infty$ for $a\geq 3$, the Berry-Esseen bound implies that the convergence rate in central limit theorem on $\log Z_n$ is at least $n^{-1/2}$. One may wonder whether there can be a  faster rate. In fact, when $\log m_0(\mathbf t)$ is a.s. non-lattice, we will see  that $n^{-1/2}$ is just the exact convergence rate. 

\begin{thm}[Exact convergence rate]\label{tt3}
	Assume that $\log m_0(\mathbf t)$ is a.s. non-lattice, $\sigma>0$, $\mathbb E (\log m_0(\mathbf t))^3<\infty$, $\mathbb E m_0(\mathbf t)^{-\alpha}<1$ and $\mathbb E \left (\frac{Y_0 }{m_0(\mathbf t)}\right)^\alpha<\infty$ for some $\alpha>0$.  If either of the following two assertions holds:
\begin{itemize}
\item[$\mathrm(i)$]$\mathbb E(\frac{e^{\mathbf t\mathbf L_1}}{m_0(\mathbf t)})^{-\alpha}<\infty$  and  $ {(H)}$ is true for some $\varepsilon >0$ and $p>\max\{1+\frac{1}{\varepsilon},\frac{2}{\varepsilon}\}$;
\item[$\mathrm(ii)$]$\mathbb E \left | \mathbf t \mathbf L_1 \right | ^a<\infty$, $\mathbb{E}(\log m_0(\mathbf t))^a <\infty$ for some $a\ge 3$, and $(H)$ is true for some $\varepsilon>0$ and $p > \max\{1+\frac{a}{(a-1)\varepsilon},\frac{2a}{(a-1)\varepsilon}\}$, 
\end{itemize}
then 
\begin{equation}\label{te3}
		\lim_{n\rightarrow\infty}\sqrt{n}\left[\mathbb P\left(\frac{\log Z_n  -n\mu }{\sqrt{n}\sigma  }\leq x\right)-\Phi(x)\right]=-\frac{1}{\sigma }\varphi(x)\mathbb E\log W+Q(x),
	\end{equation}
where  $\varphi(x)=\frac{1}{\sqrt {2\pi}} e^{-x^2/2}$ is the density function of the standard normal distribution, and $Q(x)=\frac{1}{6\sigma^3}\mu_3(1-x^2)\varphi(x)$, with $\mu_3=\mathbb E(\log m_0(\mathbf t)-\mu)^3$.
\end{thm}

Theorem \ref{tt3} shows the exact convergence in the central limit theorem, that is 
 $$\mathbb P\left(\frac{\log Z_n -n\mu}{\sqrt{n}\sigma }\leq x\right)-\Phi(x)\sim g(x)n^{-1/2}\quad (n\to\infty),$$
where the function $g(x)$ has a concrete expression that $g(x)=-\frac{1}{\sigma }\varphi(x)\mathbb E\log W+Q(x)$. This result is an generalization of \cite{G21} for BPRE and \cite{HZG} for BPIRE.

\medskip
The remainder of the paper  is arranged as follows. In Section \ref{BS2}, we recall the existence of the  positive and harmonic moments of $W_n$ and show sufficient conditions  for the existence of its logarithmic moments. Then by assuming the  existence of the logarithmic moments, we investigate the  convergence rates of $\log W_n$ in Section \ref{BS3}. These fast decay rates will play an important role in the study of   convergence rates in  limit theorems associated to $Z_n$. In Section \ref{BS4}, we present some auxiliary results which are helpful to evaluate the differences between $\log Z_n-n\mu$ and $S_n$. Finally, Section \ref{BS5} is devoted to the proofs of main theorems of the paper.

\section{Moments  of $\log W_n$}\label{BS2}	
In this section, we are interested in  logarithmic moments of $W_n$. The logarithmic moments can be controlled by positive and negative moments. We recall the results on moments of $W_n$ at first.

\begin{lem}[Positive moments \cite{WH22+}]\label{le3}
Let $\alpha>0$. Assume that $\mathbb E\left(\frac{Y_0}{m_0(\mathbf t)}\right)^\alpha<\infty$. If either of the following two conditions holds:
\begin{itemize}
\item[$\mathrm(i)$] $\alpha\in(0,1]$ and  $\mathbb Em_0(\mathbf t)^{-\alpha}<1$;
\item[$\mathrm(ii)$] $\alpha>1$,  $\max \left \{ \mathbb{E}m_0(\mathbf t)^{-\alpha},  \mathbb{E}\frac{m_0(\alpha\mathbf t)}{m_0(\mathbf t)^\alpha}\right \} <1$ and $\mathbb E\bar W_1^\alpha<\infty$,
\end{itemize}
then we have $\sup_n\mathbb E W_n^\alpha<\infty$.
\end{lem}

\begin{lem}[Harmonic moments \cite{WH22+}]\label{le2}
Assume that (H) is true for   some $\varepsilon >0$ and $p>\max\{1+\frac{1}{\varepsilon},\frac{2}{\varepsilon}\}$. Set $A_1=\frac{e^{\mathbf t\mathbf L_1}}{m_0(\mathbf t)}$. If $\mathbb E A_1^{-\alpha}<\infty$, then there exists $r>0$ such that $\sup_n\mathbb E W_n^{-r}<\infty$.
\end{lem}

Although one can use moments of $W_n$ to ensure the existence of its logarithmic moments, it is still necessary to study this question by a direct calculation. With a different approach, one may find another sufficient condition for  the existence of logarithmic moments.
Following this idea, we obtain the  theorem below on moments of $\log W_n$.

\begin{thm}[Logarithmic moments]\label{lm}
Set $A_1=\frac{e^{\mathbf t\mathbf L_1}}{m_0(\mathbf t)}$. Let $r>0$.  If $\mathbb E m_0(\mathbf t)^{-\alpha}<1$, $\mathbb E \left (\frac{Y_0 }{m_0(\mathbf t)}\right)^\alpha<\infty$ for some $\alpha>0$, and either of the following two conditions holds:
\begin{itemize}
\item[$\mathrm(i)$] $\mathbb E A_1^{-\alpha}<\infty$, and  $ {(H)}$ is true for some $\varepsilon >0$ and $p>\max\{1+\frac{1}{\varepsilon},\frac{2}{\varepsilon}\}$;
\item[$\mathrm(ii)$]$\mathbb E \left | \log A_1 \right | ^a<\infty$ for some $a>r$, and $(H)$ is true for some $\varepsilon>0$ and $p > \max\{1+\frac{a}{(a-r)\varepsilon},\frac{2a}{(a-r)\varepsilon}\}$, 
\end{itemize}
then 
$$\sup_n\mathbb E|\log W_n|^r<\infty\quad\text{and}\quad \mathbb E|\log W|^r<\infty.$$
\end{thm}

We shall prove Theorem \ref{lm} by working on the  Laplace transform of the limit $\bar W$. Define
$$\bar \phi(s)=\mathbb E e^{-s \bar W} \qquad (s>0).$$
The moments of $\bar W$ rely on the decay rates of $\bar \phi(s)$ as $s$ tends to infinity. The following lemma provides an approach to study the decay rates  of $\bar \phi(s)$.

\begin{lem}[\cite{HZG}, Lemma 2.1]\label{le1}
Let $\phi(s)$ be a bounded function satisfying
\begin{equation*} 
\phi(s)\leq q \mathbb E\phi(Gs)+\mathbb P(s<V)\qquad (\forall s\geq 0),
\end{equation*} 
where $q\in(0,1)$ is a constant, and $G, V$ are non-negative random variables. If  $\mathbb E|\log G|^r<\infty$ and $\mathbb E|\log V|^r<\infty$ for some $r>0$, then $\phi(s)=O((\log s)^{-r})$ as $s\to \infty$.
\end{lem}

With  the help of Lemma \ref{le1}, we can obtain the following result on decay rates of $\bar \phi(s)$, the  Laplace transform of $\bar W$.
\begin{pr}\label{pb}
Set $A_1=\frac{e^{\mathbf t\mathbf L_1}}{m_0(\mathbf t)}$.
Let $r>0$. If  $\mathbb E \left | \log A_1 \right | ^a<\infty$ for some $a>r$, and $(H)$ is true for some $\varepsilon>0$ and $p > \max\{1+\frac{a}{(a-r)\varepsilon},\frac{2a}{(a-r)\varepsilon}\}$, then we have $\bar \phi(s)=O((\log s)^{-r})$ as $s\to \infty$.
\end{pr}

\begin{proof}
The proof follows the idea in  the proof of \cite[Theorem 2.4]{WH22+}. We first introduce 
some notation. For $j\leq n-1$, denote
$$\Pi_{j,n}(\mathbf t)=\prod_{i=j}^{n-1}m_i(\mathbf t).$$
In particular, we have $\Pi_{0,n}(\mathbf t)=\Pi_{n}(\mathbf t)$.
Since $p>1$, there exists an integer $m\geq 0$ such that $p\in(2^m, 2^{m+1}]$. Let $T$ be the shift operator such that $T\xi=(\xi_1,\xi_2,\cdots)$ if $\xi=(\xi_0,\xi_1, \xi_2, \cdots)$. Set
$$\eta_n^{(p)}(\mathbf t)=\mathbb E_{T^n\xi}|\bar W_1(\mathbf t)-1|^p,$$
$$
x_i^{(p)}=\left\{
\begin{array}{ll}
2^i,& i=0,1,\cdots, m,\\
p,& i=m+1,
\end{array}
\right.
$$
and for $v\in\{1,2,\cdots, m+1\}$,
$$\Delta_j^{(p)}(k_1,\cdots,k_v)=\prod_{i=1}^v\frac{\Pi_{j,k_i+j}(x_{i}^{(p)} \mathbf t)^{p/x_{i}^{(p)}}}{\Pi_{j,k_i+j}(x_{i-1}^{(p)} \mathbf t)^{p/x_{i-1}^{(p)}}}\eta_{k_i+j}^{(p/x_{i-1}^{(p)})}(x_{i-1}^{(p)}\mathbf t).$$
Take $\delta$ satisfying 
$0<\delta<\frac{a-r}{a}\varepsilon$ and $p>\max\{1+\frac{1}{\delta}, \frac{2}{\delta}\}$.  Set    $K=K_n=n^\gamma$, where $\gamma$ satisfies $\frac{1}{\delta}<\gamma<\frac{(\tilde p-1)p}{\tilde p}$, with $\tilde p=\min\{p,2\}$. Let  $(N(l), A_1(l),A_2(l),\cdots)$, $l=0, 1,2,\cdots$, be  a sequence of random vectors   which is conditional independent  when given the environment $\xi$, and with distributions $$\mathbb P_\xi((N(l), A_1(l),A_2(l),\cdots)\in\cdot)=\mathbb P_{T^l\xi}((N, A_1,A_2,\cdots)\in\cdot).$$ 
Set $A_1(-1)=1$ by convention. 
For constant $b>0$, denote $N^{(b)}(l)=\mathbf{1}_{\{N(l)\geq 2\}}\sum_{i=2}^{N(l)}\mathbf{1}_{\{A_i(l)>b\}}$. According to the proof of \cite[Theorem 2.4]{WH22+} (refer to (5.14)), we have for all $b>0$ and $n> 1$,
\begin{eqnarray}\label{4tpe10}
\bar \phi(s)&\leq &\mathbb P(s<V)+\mathbb E\left[\bar\phi\left(s\prod_{l=0}^{n-1}A_1(l)\right)\prod_{j=0}^{n-1}\left(\mathbf{1}_{\{N(j)=1\}}+\mathbf{1}_{\{N(j)\geq 2, N^{(b)}(j)=0\}}\right.\right.\notag\\
&&\left.\left.+ B_K\mathbf{1}_{\{N(j)\geq 2, N^{(b)}(j)\geq 1\}}
\right)\right]
+\frac{c_p^\delta}{K^\delta}\sum_{j=0}^{n-1}\left[\left(\mathbb E\bar \phi\left(s\prod_{l=0}^{n-1}A_1(l)\right)\right)^{1/\beta}\right.\notag\\
&&\left.+\sum_{v=1}^{m+1}\sum_{k_1=0}^{n-j-2}\cdots\sum_{k_v=0}^{k_{v-1}-1}\left(\mathbb E\bar \phi\left(s\prod_{l=0}^{n-1}A_1(l)\right)\Delta_{j+1}^{(p)}(k_1,\cdots, k_v)^{\delta}\right)^{1/\beta}\right.\;\;\notag\\
&&\left.+\sum_{v=1}^{m+1}\sum_{k_1=n-j-1}^{\infty}\cdots\sum_{k_v=0}^{k_{v-1}-1}\left(\mathbb E\bar \phi\left(s\prod_{l=0}^{j+k_1+1}A_1(l)\right)\Delta_{j+1}^{(p)}(k_1,\cdots, k_v )^{\delta}\right)^{1/\beta}\right]^\beta.\;\;\;\;\;\;\;\;\;\;\;
\end{eqnarray}
where $\beta=\max\{\frac{p}{2},1\}\delta$, $c_p$ is a positive constant depending just on $p$, 
$$B_K=1-c_pK^{-\frac{\tilde p}{(\tilde p-1)p}}\in(0,1),$$
and
$$V=S_{p,b,K}\max_{0\leq j\leq n-1}\{\prod_{l=0}^{j-1}A_1(l)^{-1},1\},$$
with $S_{p,b,K}$ a positive constant depending on $p,b,K$. Let us   define a random variable $G$ as follows: for any measurable bounded function  $g$,
\begin{eqnarray*}
\mathbb E g(G)&=&\frac{1}{q}\left\{\mathbb E\left[ g\left(\prod_{l=0}^{n-1}A_1(l)\right)\prod_{j=0}^{n-1}\left(\mathbf{1}_{\{N(j)=1\}}+\mathbf{1}_{\{N(j)\geq 2, N^{(b)}(j)=0\}}\right.\right.\right.\notag\\
&&\left.\left.+ B_K\mathbf{1}_{\{N(j)\geq 2, N^{(b)}(j)\geq 1\}}
\right)\right]
+\frac{c_p^\delta}{K^\delta}\sum_{j=0}^{n-1}\left[\left(\mathbb Eg\left(\prod_{l=0}^{n-1}A_1(l)\right)\right)^{1/\beta}\right.\notag\\
&&\left.+\sum_{v=1}^{m+1}\sum_{k_1=0}^{n-j-2}\cdots\sum_{k_v=0}^{k_{v-1}-1}\left(\mathbb E g\left(\prod_{l=0}^{n-1}A_1(l)\right)\Delta_{j+1}^{(p)}(k_1,\cdots, k_v)^{\delta}\right)^{1/\beta}\right.\;\;\notag\\
&&\left.\left.+\sum_{v=1}^{m+1}\sum_{k_1=n-j-1}^{\infty}\cdots\sum_{k_v=0}^{k_{v-1}-1}\left(\mathbb E g\left(\prod_{l=0}^{j+k_1+1}A_1(l)\right)\Delta_{j+1}^{(p)}(k_1,\cdots, k_v)^{\delta}\right)^{1/\beta}\right]^\beta\right\},
\end{eqnarray*}
where $q$ is the normalization constant  
that 
\begin{eqnarray*}
q&=&\mathbb E\left(\prod_{j=0}^{n-1}\left(\mathbf{1}_{\{N(j)=1\}}+\mathbf{1}_{\{N(j)\geq 2, N^{(b)}(j)=0\}}+ B_K\mathbf{1}_{\{N(j)\geq 2, N^{(b)}(j)\geq 1\}}
\right)\right)\\
&&+\frac{c_p^\delta}{K^\delta}\sum_{j=0}^{n-1}\left(1+\sum_{v=1}^{m+1}\sum_{k_1=0}^{\infty}\cdots\sum_{k_v=0}^{k_{v-1}-1}\left(\mathbb E \Delta_{j+1}^{(p)}(k_1,\cdots, k_v)^{\delta}\right)^{1/\beta}\right)^\beta.
\end{eqnarray*}
We can calculate  
$$q=\rho_{K,b}^n+C_{p,\delta}\frac{n}{K^\delta},$$
where
$$\rho_{K,b}=\mathbb P(N=1)+\mathbb P(N\geq 2,N^{(b)}=0)+B_K(N\geq 2,N^{(b)}\geq 1)$$
and
$$C_{p,\delta}=c_p^\delta\left(1+\sum_{v=1}^{m+1}\sum_{k_1=0}^{\infty}\cdots\sum_{k_v=0}^{k_{v-1}-1}\left(\mathbb E \Delta_{0}^{(p)}(k_1,\cdots, k_v)^{\delta}\right)^{1/\beta}\right)^\beta<\infty.$$
Thus, we can write \eqref{4tpe10} as
\begin{equation*} 
\bar \phi(s)\leq q \mathbb E\bar \phi(Gs)+\mathbb P(s<V).
\end{equation*}
Notice that 
$$q=\rho_{K,b}^n+C_{p,\delta}\frac{n}{K^\delta}\stackrel{b\downarrow0}{\longrightarrow}\rho_{K}^n+C_{p,\delta}\frac{n}{K^\delta}\stackrel{n\uparrow\infty}{\longrightarrow}0,$$
where $\rho_K=\mathbb P(N=1)+B_K\left(1-\mathbb P(N=1)\right)$.  Thus, we can choose appropriate  $n, b$ such that $q\in(0,1)$. In order to derive $\bar \phi(s)=O((\log s)^{-r})$, by Lemma \ref{le1}, it remains to show that $\mathbb E|\log V|^r<\infty$ and
$\mathbb E |\log G|^r<\infty$.

Since $\mathbb E|\log A_1|^a<\infty$ for some $a>r$, we can calculate that 
\begin{eqnarray*}
\mathbb E|\log V|^r&=&\mathbb E\left|\log \left(S_{p,b,K}\max_{0\leq j\leq n-1}\{\prod_{l=0}^{j-1}A_1(l)^{-1},1\}\right) \right|^r\\
&\le& C\left(|\log S_{p,b,K}|^r+\mathbb E|\log A_1|^r\sum_{j=1}^{n-1}j\right)<\infty,
\end{eqnarray*}
where here and throughout the paper $C>0$ represents a general constant (it may depend on $\mathbf t$).
For $\mathbb E |\log G|^r$, we can calculate that 
\begin{eqnarray*}
q\mathbb E |\log G|^r&\le&\mathbb E \left|\sum_{l=0}^{n-1}\log A_1(l)\right|^r
+\frac{c_p^\delta}{K^\delta}\sum_{j=0}^{n-1}\left[\left(\mathbb E\left|\sum_{l=0}^{n-1}\log A_1(l)\right|^r\right)^{1/\beta}\right.\notag\\
&&\left.+\sum_{v=1}^{m+1}\sum_{k_1=0}^{n-j-2}\cdots\sum_{k_v=0}^{k_{v-1}-1}\left(\mathbb E \left|\sum_{l=0}^{n-1}\log A_1(l)\right|^r\Delta_{j+1}^{(p)}(k_1,\cdots, k_v)^{\delta}\right)^{1/\beta}\right.\;\;\notag\\
&&\left.+\sum_{v=1}^{m+1}\sum_{k_1=n-j-1}^{\infty}\cdots\sum_{k_v=0}^{k_{v-1}-1}\left(\mathbb E \left|\sum_{l=0}^{j+k_1+1}\log A_1(l)\right|^r\Delta_{j+1}^{(p)}(k_1,\cdots, k_v)^{\delta}\right)^{1/\beta}\right]^\beta.
\end{eqnarray*}
Notice that for any $k\geq 0$,
$$\mathbb E\left|\sum_{l=0}^{k}\log A_1(l)\right|^r\leq C (k+1)\mathbb E|\log A_1|^r<\infty.$$
To reach $q\mathbb E |\log G|^r<\infty$, we just need to show that for any $j\in\{0,1,\dots, n-1\}$ and $v\in\{1,2,\cdots, m+1\}$,
\begin{equation}\label{de}
\sum_{k_1=n-j-1}^{\infty}\cdots\sum_{k_v=0}^{k_{v-1}-1}\left(\mathbb E \left|\sum_{l=0}^{j+k_1+1}\log A_1(l)\right|^r\Delta_{j+1}^{(p)}(k_1,\cdots, k_v)^{\delta}\right)^{1/\beta}<\infty.
\end{equation}
Set $\tilde \delta=\frac{a}{a-r}\delta$. Then $0<\tilde \delta<\varepsilon$.
Since $\mathbb E\left(\frac{m_0(p\mathbf t)}{m_0(\mathbf t)^p}\right)^{\tilde\delta}<1$ and $\mathbb E\left(\mathbb E_\xi \bar W_1\right)^{\tilde\delta}<\infty$, according the proof of \cite[Theorem 2.4]{WH22+},   there exists a constant $\rho(p,\tilde\delta)\in(0,1)$ depending on $(p,\tilde\delta)$ such that 
$$\mathbb E \Delta_{0}^{(p)}(k_1,\cdots, k_v)^{\tilde \delta}\leq C \rho(p,\tilde\delta)^{k_1-v}.$$
 By H\"older's inequality, 
\begin{eqnarray*}
&&\mathbb E \left|\sum_{l=0}^{j+k_1+1}\log A_1(l)\right|^r\Delta_{j+1}^{(p)}(k_1,\cdots, k_v)^{\delta}\\
&\le& (j+k_1+1)^{r-1}\sum_{l=0}^{j+k_1+1}\mathbb E\left[\left|\log A_1(l)\right|^r\Delta_{j+1}^{(p)}(k_1,\cdots, k_v)^{\delta}\right]\\
&\le&(j+k_1+1)^{r-1}\sum_{l=0}^{j+k_1+1}\left(\mathbb E\left|\log A_1(l)\right|^a\right)^{r/a}\left(\mathbb E\Delta_{j+1}^{(p)}(k_1,\cdots, k_v)^{\tilde \delta}\right)^{1-r/a}\\
&\le &C(j+k_1+1)^{r} \rho(p,\tilde\delta)^{(k_1-v)(1-r/a)},
\end{eqnarray*}
which leads to \eqref{de}.
\end{proof}

Now let us give the proof of Theorem \ref{lm}, by using the moments of $W_n$  and  the decay rates of $\bar \phi(s)$.

\begin{proof}[Proof of Theorem \ref{lm}]
Since $\mathbb E m_0(\mathbf t)^{-\alpha}<1$ and $\mathbb E \left (\frac{Y_0 }{m_0(\mathbf t)}\right)^\alpha<\infty$ for some $\alpha>0$, by Lemma \ref{le3}, there exists $\epsilon>0$ small enough such that  $\sup_n \mathbb E W_n^\epsilon<\infty$. Suppose that the condition (i) is satisfied. Then it follows from Lemma \ref{le2} that $\sup_n \mathbb E W_n^{-\epsilon}<\infty$  for $\epsilon>0$ small enough. Noticing that $|\log x|\leq C(x^{\epsilon}+x^{-\epsilon})$, we  derive
$$\sup_n \mathbb E |\log W_n|^r\leq C(\sup_n \mathbb E W_n^{\epsilon}+\sup_n \mathbb E W_n^{-\epsilon})<\infty.$$

Now suppose that the condition (ii) is satisfied. Notice that
\begin{eqnarray}\label{pt4e}
\mathbb E|\log W_n|^r&=&\mathbb E|\log W_n|^r\mathbf{1}_{\{W_n\geq 1\}}+\mathbb E|\log W_n|^r\mathbf{1}_{\{W_n\le 1\}}\notag \\
&\leq &C\sup_n\mathbb E W_n^\epsilon+\sup_n\mathbb E|\log W_n|^r\mathbf{1}_{\{W_n\le 1\}}.\;\;\;\;
\end{eqnarray}
For  $x>0$, define
$$
g(x)=\left\{\begin{array}{ll}
(-\log x)^r & \text{if $0<x<x_r$;}\\
(-\log x_r)^r & \text{if $x\geq x_r$,}
\end{array}
\right.
$$
where $x_r=\min\{e^{r-1},1\}$. The function  $g(x)$ is   decreasing and convex  on $(0,\infty)$ satisfying $g(x)\geq |\log x|^r\mathbf{1}_{\{x< 1\}}$ for $x>0$. Thus,
\begin{equation}\label{p1e1}
\sup_n\mathbb E|\log W_n|^r\mathbf{1}_{\{W_n< 1\}}
\leq \sup_n\mathbb E g(W_n)\leq \sup_n\mathbb E g(\bar W_n)=\mathbb E g(\bar W),
\end{equation}
where we have used the monotonicity and convexity of $g(x)$ as well as \cite[Lemma 2.1]{hl12}. It can be seen that
\begin{eqnarray}\label{p1e2}
\mathbb E g(\bar W)&=&\mathbb E(-\log \bar W)^r\mathbf{1}_{\{\bar W< x_r\}}+(-\log x_r)^r\mathbb P(\bar W\ge x_r)\notag \\
&\le&\mathbb E(-\log \bar W)^r\mathbf{1}_{\{\bar W< x_r\}}+(-\log x_r)^r.
\end{eqnarray}
Recall that $\bar \phi(s)=\mathbb E e^{-s \bar W}$. Take $r_1$ satisfying $r<r_1<a$. By Proposition \ref{pb}, we have $\bar \phi(s)=O((\log s)^{-r_1})$. Therefore, we have
\begin{eqnarray}\label{ple3}
\mathbb E(-\log \bar W)^r\mathbf{1}_{\{\bar W< x_r\}}
&\le& C+ r\int_{x_r^{-1}}^\infty s^{-1}(\log s)^{r-1}\mathbb P(\bar W<s^{-1})\mathrm{d}s\nonumber \\
&\leq &C\left(1+  \int_{x_r^{-1}}^\infty s^{-1}(\log s)^{r-1}\bar \phi(s)\mathrm{d}s\right)<\infty.
\end{eqnarray} 
Combining \eqref{pt4e}-\eqref{ple3} gives $\sup_n \mathbb E |\log W_n|^r<\infty$. 
\end{proof}

		\section{Convergence rates of $\log W_n$}\label{BS3}
		In this section, we shall work on convergence rates of $\log W_n$, by assuming the existence of the logarithmic moments.  First, we see that  the submartingale $W_n$ can converge in $L^\alpha$ ($\alpha\in(0,1]$) very fast, with an exponential rate, under appropriate conditions.
			 		
\begin{lem}\label{s3l01}
	Let $\alpha\in(0,1]$. Assume that $\mathbb{E}m_0(\mathbf t)^{-\alpha}<1$ and $\mathbb{E}\left(\frac{Y_0}{m_0(\mathbf t)}\right)^\alpha<\infty$,  and $(H)$ is true for some $\varepsilon>0$ and $p\in (1,2]$ satisfying $p\varepsilon>\alpha$. Then there exists a constant $\rho\in (0,1)$ such that for $l>n$,
	\begin{equation}\label{s301}
		\mathbb E\left | W_l -W_n \right |^\alpha \le c\rho^n .
	\end{equation}
\end{lem}
	\begin{proof}
Notice that
		\begin{equation}\label{s302}
				\mathbb E\left | W_l -W_n \right |^\alpha \le  \sum_{k=n}^{l-1} \mathbb E\left | W_{k+1} -W_k \right |^\alpha.
		\end{equation}
According to the structure of the family tree $\mathbb{T}$,  we can decompose $W_k$ as follows: 
\begin{equation} \label{CRE2.2.23}
W_k=\bar W_k+\sum_{j=1}^{k} \Pi_j(\mathbf  t)^{-1}\sum_{i=1}^{V_{j-1}}e^{\mathbf t\mathbf S_{0_{j-1}i}}\bar W_{k-j}^{(0_{j-1}i)},
\end{equation}
where $\bar W_k^{(u)}$ denotes the Biggins martingale of the BRWRE originating from the particle $u\in\mathbb U$. 
By the above decomposition, we have
	\begin{equation*}
		W_{k+1}-W_{k}	 = \bar W_{k+1}-\bar W_k +\sum_{j=1}^{k}\Pi_j(\mathbf t)^{-1} \sum_{i=1}^{V_{j-1}}e^{\mathbf t\mathbf S_{0_{j-1}i}} (\bar W_{k+1-j}^{(0_{j-1}i)}-\bar W_{k-j}^{(0_{j-1}i)}) + \Pi_{k+1}(\mathbf t)^{-1}Y_k.
	\end{equation*}
Therefore,  realizing that $\mathbb E\left (\frac{Y_0}{m_0(\mathbf t)}\right)^\alpha<\infty$, we see that
		\begin{eqnarray}\label{s303}
\mathbb E \left | W_{k+1}-W_k \right |^\alpha&\leq&\mathbb E(\mathbb E_{\xi, Y}\left | W_{k+1}-W_k \right |)^\alpha\nonumber\\
	&\leq&\mathbb E(\mathbb E_{\xi}\left | \bar W_{k+1}-\bar W_k \right |)^\alpha\nonumber\\ 
	&\quad&+\sum_{j=1}^{k}\mathbb{E}\left[\Pi_j(\mathbf t)^{-1}\mathbb E_{\xi, Y}\left(\sum_{i=1}^{V_{i-1}}e^{\mathbf t\mathbf S_{v_{j-1}i}}\left |\bar W_{k+1-j}^{0_{j-1}i)}-\bar W_{k-j}^{(0_{j-1}i)}\right |\right )\right]^\alpha\nonumber\\
	&\quad&+(\mathbb E m_0(\mathbf t)^{-\alpha})^k\mathbb{E}\left (\frac{Y_0}{m_0(\mathbf t)}\right)^\alpha\nonumber\\
	&\leq&\mathbb E(\mathbb E_{\xi}\left | \bar W_{k+1}-\bar W_k \right |)^\alpha\nonumber\\
	&\quad&+\mathbb {E}
	\left (\frac{Y_0 }{m_0(\mathbf t)}\right)^\alpha \left(  \mathbb{ E}m_0(\mathbf t) ^{-\alpha }  \right)^{j-1}\sum_{j=1}^{k}\mathbb{E}\left( \mathbb E_\xi \left | 
	\bar W_{k+1-j}-\bar W_{k-j} \right |  \right) ^\alpha \notag\\
	&\quad&+ \mathbb E\left (\frac{Y_0}{m_0(\mathbf t)}\right)^\alpha
	(\mathbb E m_0(\mathbf t)^{-\alpha})^{ k}\nonumber\\
	&\leq&C\sum_{j=0}^{k+1}\left(  \mathbb{ E}m_0(\mathbf t) ^{-\alpha }  \right)^{j-1}\mathbb E\left ( \mathbb E_\xi 
	\left | \bar W_{k+1-j} - \bar W_{k-j} \right |  \right ) ^\alpha,
\end{eqnarray}
where we set $\bar W_{-1}=0$ by convention. We can calculate for $0\le j\le k$,
\begin{eqnarray*}
	\mathbb E_\xi\left | \bar W_{k+1-j}-\bar W_{k-j} \right |&\leq&
	\left( \mathbb E_\xi\left|\bar W_{k+1-j}-\bar W_{k-j} \right|^p\right)^{{1}/{p}}\\
	&\leq&C\left ( \frac{\Pi_{k-j}(p\mathbf t)}{\Pi_{k-j}(\mathbf t)^p}  \right ) ^{ {1}/{p} }\left ( \mathbb E_{T^{k-j}\xi }
	\left | \bar W_1-1 \right |^p  \right ) ^{{1}/{p} }.
\end{eqnarray*}
Thus
\begin{eqnarray}\label{s304}
\mathbb E\left ( \mathbb E_\xi 
\left | \bar W_{k+1-j} - \bar W_{k-j} \right |  \right ) ^\alpha&\leq&C\mathbb E\left ( \frac{\Pi_{k-j}(p\mathbf t)}{\Pi_{k-j}(\mathbf t)^p}  \right ) ^{{\alpha}/{p} }\mathbb E\left ( \mathbb E_\xi
\left | \bar W_1-1 \right |^p  \right ) ^{{\alpha}/{p} }\nonumber\\
	&\leq&C\left [ \mathbb{E}\left ( \frac{m_0(p\mathbf t)}{m_0(\mathbf t)^p}  \right )^{{\alpha }/{p} }   \right ] ^{k-j},
\end{eqnarray}
since $\mathbb E (\mathbb E_\xi \bar W_1^p)^{\frac{\alpha}{p}}\le \left [ \mathbb E(\mathbb E_\xi \bar W_{1}^p)^\varepsilon  \right ]^{\frac{\alpha }{p\varepsilon }}<\infty$. Combining \eqref{s302}, \eqref{s303} and \eqref{s304}  yields
\begin{eqnarray}\label{s305}
	\mathbb E\left | W_l-W_n \right |^\alpha&\leq&C\sum_{k=n}^{l-1} \left [ \sum_{j=0}^{k}(\mathbb Em_0(\mathbf t)^{-\alpha })^{j-1}\left ( \mathbb E\left (  \frac{m_0(p\mathbf t)}{m_0(\mathbf t)^p} \right )^{{\alpha }/{p} }
	\right )  ^{k-j}+(\mathbb Em_0(\mathbf t)^{-\alpha}) ^k\right] \nonumber\\
	&\leq&C\sum_{k=n}^{l-1}\sum_{j=0}^{k+1} \rho_ {\alpha} ^{k-1},
\end{eqnarray}
where $\rho_\alpha = \max\{\mathbb Em_0(\mathbf t)^{-\alpha},\mathbb E\left ( \frac{m_0(p\mathbf t)}{m_0(\mathbf t)^p}  \right )^{{\alpha }/{p } } \}<1$. since $\alpha<p\varepsilon$.  Take $\rho\in(\rho_\alpha,1)$.  It follows that
\begin{eqnarray}\label{s306}
	\sum_{k=n}^{l-1}\sum_{j=0}^{k+1}\rho_\alpha ^{k-1} \leq C\rho_\alpha ^{-1}\sum_{k=n}^{l-1}k\rho_\alpha ^k \leq  C\rho^n.
\end{eqnarray}	
Combining \eqref{s305} and \eqref{s306},   we  obtain  \eqref{s301}.
		\end{proof}
		
Under the conditions of Lemma \ref{s3l01}, we can further deduce the convergence rates of $\log W_n$  if the logarithmic moment	$\sup_n \mathbb E |\log W_n|^r$ exists for some $r>0$.	

For $l>n$, set					
$$\eta _{n.l}= \frac{W_l}{W_n} -1.$$ 
Then $$\log W_l-\log W_n = \log(\eta_{n.l}+1).$$ We have  the following conclusion which describes the decay rates of $\eta_{n,l}$ and $\log W_l-\log W_n$ in probability.

	\begin{pr}\label{s3l02}
Under the conditions of Lemma \ref{s3l01}, if $\sup _n\mathbb E\left | \log W_n \right |^r<\infty $ for some $r>0$, then  there exists a constant $\rho \in(0,1)$ such that for all $x>0$ and $l>n$,
	\begin{equation}\label{s307}
	\mathbb P(|\eta_{n.l}|>x)\leq C(x^{-\alpha}\rho^n +n^{-r}),
	\end{equation}
		\begin{equation}\label{s308}
		\mathbb P(|\log W_l-\log W_n|>x)\leq C\left [ (x^{-\alpha }+1)\rho ^n+n^{-r} \right ] .
	\end{equation}
\end{pr}
	\begin{proof}
	First we prove \eqref{s307}. By Lemma \ref{s3l01},  there exists $\rho_1\in(0,1)$ such that 
\begin{equation}\label{s3010}
	\mathbb E|W_l-W_n|^\alpha\leq C\rho_1^n.
\end{equation}
Take $b\in (0,1)$ satisfying $b^{-\alpha}\rho_1<1$, and set $\rho=b^{-\alpha}\rho_1$.
	By Markov's inequality,  
	\begin{eqnarray}\label{s309}
		\mathbb P(|\eta_{n.l}|>x) &=& \mathbb P(|W_l-W_n|>W_nx)\nonumber\\
		&\leq&\mathbb P(|W_l-W_n|>xb^n)+\mathbb P(W_n<b^n)\nonumber\\
		&\leq&\mathbb P(|W_l-W_n|>xb^n)+\mathbb P\left(|\log W_n|>n|\log b|\right)\nonumber\\
		&\leq&\frac{\mathbb E\left | W_l-W_n \right |^\alpha  }{x^\alpha b^{\alpha n}}+\frac{ \sup_n\mathbb E\left | 
			\log W_n\right |^r   }{|\log b|^rn^r}.
	\end{eqnarray}
	 Combining \eqref{s309} and \eqref{s3010} gives \eqref{s307}, since $\sup_n\mathbb E|\log W_n|^r<\infty$.
	  
Then let us prove \eqref{s308}.  Notice that $|\log (x+1)|\leq M|x|$ on $(-c,\infty)$ for some positive constants $M$ and $c$. By \eqref{s307},
\begin{eqnarray*}
	\mathbb P(|\log W_l-\log W_n|>x) &\le& \mathbb P (|\log (\eta_{n.l}+1)|>x,\eta_{n.l}\geq-c)+\mathbb P(\eta_{n.l}\leq-c)\\
	&\leq&\mathbb P\left(|\eta_{n.l}|>\frac{x}{M}\right)+\mathbb P(|\eta_{n.l}|>c)\\
	&\leq&C\left [ \left ( \frac{x}{M}  \right )^{-\alpha }\rho ^n+n^{-r}+c^{-\alpha }\rho^n+n^{-r}  \right ]\\
	&\leq&C\left[(x^{-\alpha}+1)\rho^n+n^{-r}\right],
\end{eqnarray*}
which gives  \eqref{s308}.
	\end{proof}
	
Moreover, it is also possible to characterize the $L^q$ $(q>0)$ decay rates of $\log W_l-\log W_n$. The following result indicates that the  $L^q$ decay rates can be polynomial based on the conditions of Lemma \ref{s3l01} and the logarithmic moment of $W_n$.
	
\begin{pr}\label{s3l03}
	Let $q>0$.  Under the conditions of Lemma \ref{s3l01}, if  $\sup_n\mathbb E|\log W_n|^{\max\{q,1\}r}<\infty$ for some $r>1$, then we have for $l>n$, 
	\begin{equation}\label{s3011}
		\mathbb E|\log W_l-\log W_n|^q =O(n^{-(r-1)}).
	\end{equation}
\end{pr}

	\begin{proof}
Take $b\in (0,1)$, where the value of $b$ will be determined later.   We have
	\begin{eqnarray}\label{s3012}
  		\mathbb E|\log W_l-\log W_n|^q&=&	\mathbb E|\log W_l-\log W_n|^q\mathbf 1_{\{|\log W_l-\log W_n|<b^n\}}\nonumber\\
  		&\quad&+\mathbb E|\log W_l-\log W_n|^q\mathbf 1_{\{|\log W_l-\log W_n|\geq b^n\}}\nonumber\\
  		&\leq&b^{nq}+\mathbb E(|\log W_l-\log W_n|^q)\mathbf 1_{\{|\log W_l-\log W_n|> b^n\}}\nonumber.
	\end{eqnarray}
Using H\"older's inequality and   \eqref{s308}, we deduce that
\begin{eqnarray}\label{s3013}
	&\quad&\mathbb E(|\log W_l-\log W_n|^q)\mathbf 1_{\{|\log W_l-\log W_n|> b^n\}}\nonumber\\
	&\leq&\mathbb E(|\log W_l-\log W_n|^{qr})^{{1}/{r}}\mathbb P( |\log W_l-\log W_n|\geq b^n)^{1-{1}/{r}}\nonumber\\
	&\leq& C\left ( \sup_n\mathbb E|\log W_n|^{qr} \right )^{{1}/{r} }\mathbb P(|\log W_l-\log W_n|>b^n)^{1-{1}/{r} } \nonumber\\
	&\leq&C(b^{-\alpha n}\rho^n+\rho^n+n^{-r})^{1-{1}/{r}}\nonumber\\
	&\leq&Cn^{-(r-1)},
\end{eqnarray}
if we take $b\in(0,1)$ satisfying $b^{-\alpha}\rho\in(0,1)$. Combining \eqref{s3012} with \eqref{s3013} yields \eqref{s3011}.
	\end{proof}


\section{Auxiliary results}\label{BS4}
In this section, we present some auxiliary results which will play an key role in the proofs of theorems.
Recall that 
\begin{equation*}
S_n=\log \Pi_n(\mathbf t)-n\mu
\end{equation*}
is the corresponding centred random walk of the system, and 
\begin{equation*}
F_n(x)=\mathbb P \left(\frac{S_n}{\sqrt n\sigma}\le x \right)\quad \left(x \in \mathbb R \right)
\end{equation*}
is the distribution function of $\frac{S_n}{\sqrt n\sigma}$.
Since $S_n$ is the partial sum of i.i.d. sequence $\left\{\log m_i(\mathbf t)-\mu\right\}$, it is known the following results about the convergence rates in central limit theorem on $S_n$  according to classical knowledge in probability theory.

\begin{lem}
\label{s4l1}
		Let $X$ be a random variable with $\mathbb E X=\mu$. Denote $\sigma^2=\mathbb E(X-\mu)^2$ and $\mu_3=\mathbb E(X-\mu)^3$. Assume that $\sigma>0$.		
Let $X_n$ be independence copies of $X$ and set $S_n=\sum_{k=1}^n(X_k-\mu)$. 
		\begin{itemize}
			\item[(a)]\emph{(Non-uniform Berry-Esseen bound, \cite{B66, P95})} If $\mathbb E|X|^{a}<\infty$ for some $a>2$,
			then
			$$\left|\mathbb P\left(\frac{S_n  }{\sqrt{n}\sigma }\leq x\right)-\Phi(x)\right|\leq Cn^{-\delta/2}(1+|x|)^{-a} ,$$
			 where $\delta=\min\{a-2,1\}$;
			\item[(b)]\emph{(Exact convergence rate, \cite{L2})}
			If $X$ is non-lattice and $\mathbb E|X|^3<\infty$, then
			$$\lim_{n\rightarrow\infty}\sqrt{n}\sup_{x\in\mathbb R}\left\{\left|\mathbb P\left(\frac{S_n }{\sqrt{n}\sigma }\leq x\right)-\Phi(x)-\frac{1}{\sqrt{n}}Q(x)\right|\right\}=0,$$
		\end{itemize}
		where $\Phi(x)=\frac{1}{\sqrt {2\pi}}\int_{-\infty}^x e^{-t^2/2}\mathrm{d}t$ is   the standard normal distribution function,  $\varphi(x)=\frac{1}{\sqrt {2\pi}} e^{-x^2/2}$   the density function of the standard normal distribution, and $Q(x)=\frac{1}{6\sigma^3}\mu_3(1-x^2)\varphi(x)$. 
\end{lem}
	
\begin{lem}\label{s4l03}
Assume that $\mathbb E\left(\log m_0(\mathbf t)\right)^{a} < \infty$ for some $a>2$ and $\sup_n \mathbb E \left|\log W_n\right|^r < \infty$ for some $r>0$. Let $v\in(0,1]$ be a constant satisfying $v<r$. Set   $\delta=\min\{a-2,1\}$, $m=[\sqrt{n} ] $ and $a_n=n^{-{\delta }/(2v) } $.
Then for all $x \in \mathbb R$, we have 
\begin{equation}\label{s401}
\mathbb P\left(\frac{\log W_m }{\sqrt{n }\sigma } +\frac{S_n}{\sqrt{n }\sigma } \ge x-a_n, \frac{S_n}{\sqrt{n }\sigma }\le x  \right) \le Cn^{-{\delta_v }/{2} },
\end{equation}
\begin{equation}\label{s402}
\mathbb P\left(\frac{\log W_m }{\sqrt{n }\sigma } +\frac{S_n}{\sqrt{n }\sigma } \le x+a_n, \frac{S_n}{\sqrt{n }\sigma }\ge x  \right) \le Cn^{-{\delta_v }/{2} },
\end{equation}
where $\delta_v=\min\{a-2,v\}$.
\begin{proof}
We just prove  \eqref{s401}, and the proof of \eqref{s402} is similar.   By Lemma \ref{s4l1}(a),
\begin{equation}\label{s403}
\sup_{x \in \mathbb R}\left|F_n(x)-\Phi (x)\right| \le Cn^{-{\delta }/{2} } .
\end{equation}
Notice that
\begin{equation*}
\log W_m +S_n = \log W_m +S_m+(S_n-S_m).
\end{equation*}
Let $\nu_n(\mathrm{d}y,\mathrm{d}z)=\mathbb P\left(\frac{S_m}{\sqrt{n }\sigma }  \in \mathrm{d}y,\frac{\log W_m }{\sqrt{n }\sigma }\in \mathrm{d}z \right)$. Then
\begin{eqnarray}\label{s404}
&\quad&\mathbb P\left(\frac{\log W_m }{\sqrt{n }\sigma } +\frac{S_n}{\sqrt{n }\sigma } \ge x-a_n, \frac{S_n}{\sqrt{n }\sigma }\le x  \right)\nonumber\\
&=&\mathbb P\left(\frac{\log W_m }{\sqrt{n }\sigma } +\frac{S_m}{\sqrt{n }\sigma } +\frac{(S_n-S_m)}{\sqrt{n }\sigma }\ge x-a_n, \frac{S_m}{\sqrt{n }\sigma } +\frac{(S_n-S_m)}{\sqrt{n }\sigma }\le x  \right) \nonumber\\
&=&\int \mathbf 1_{\left\{a_n+z \ge0\right\}}\left[F_{n-m}\left( A\right)-F_{n-m}\left(B\right)\right]\nu_n(\mathrm{d}y,\mathrm{d}z),
\end{eqnarray}
where $A=\frac{\sqrt{n} }{\sqrt{n-m} }(x-y)$ and $B=\frac{\sqrt{n} }{\sqrt{n-m} }(x-a_n-y-z)$.
By \eqref{s403}, and using the mean value theorem, 
\begin{eqnarray}\label{s405}
&\quad&\left|F_{n-m}\left(A \right)-F_{n-m}\left(B\right)\right|\nonumber\\
&\le& 2\sup_{x \in \mathbb R}\left|F_{n-m}(x)-\Phi (x)\right|+\left|\Phi \left(A\right)-\Phi\left(B\right)\right|\nonumber\\
&\le& Cn^{-{\delta }/{2} } +\min\left\{1,\;\sup_{x \in \mathbb R}|\varphi (x)|\frac{\sqrt{n} }{\sqrt{n-m} }|a_n+z|\right\}\nonumber\\
&\le& C(n^{-{\delta }/{2} }+|z|^v),
\end{eqnarray}
where   we have used the fact that $\sup_{x \in \mathbb R}|\varphi(x)|\le C$.
Combining \eqref{s404} with \eqref{s405} yields
\begin{eqnarray*}
&\quad&\mathbb P\left(\frac{\log W_m }{\sqrt{n }\sigma} +\frac{S_n}{\sqrt{n }\sigma} \ge x-a_n, \frac{S_n}{\sqrt{n }\sigma}\le x \right ) \\
&\le& C\left(n^{-{\delta }/{2} }+\int |z|^v\nu_n(\mathrm{d}y,\mathrm{d}z)\right)\\
&\le& C\left(n^{-{\delta }/{2} }+n^{-v/2} \sup_n \mathbb E\left|\log W_n\right|^v\right)\\
&\le& Cn^{-{\delta_v }/{2} },
\end{eqnarray*}
which gives \eqref{s401}.
\end{proof}
\end{lem}

\begin{lem}	\label{sl403}
Assume that $\mathbb E\left(\log m_0(\mathbf t)\right)^{a} < \infty$ for some $a>2$ and $\sup_n \mathbb E\left|\log W_n\right|^{r} < \infty$ for some $r>0$.  Let $v\in(0,1]$ be a constant satisfying $v<r$.  Set   $\delta=\min\{a-2,1\}$,  $m=\left[\sqrt n\right]$ and $\alpha _n = n^{-{\delta }/{(2v)} }|x|^\beta$ with $\beta \in \left(0,1\right)$. Then for $|x|\ge 1$, we have 
\begin{equation}\label{s406}
\mathbb P\left(\frac{\log W_m }{\sqrt{n }\sigma } +\frac{S_n}{\sqrt{n }\sigma } \ge x-\alpha_n, \frac{S_n}{\sqrt{n }\sigma }\le x  \right) \le Cn^{-{\delta_v }/{2} }(1+|x|)^{-\lambda } ,
\end{equation}
\begin{equation}\label{s407}
\mathbb P\left(\frac{\log W_m }{\sqrt{n }\sigma } +\frac{S_n}{\sqrt{n }\sigma} \le x+\alpha_n, \frac{S_n}{\sqrt{n }\sigma }\ge x \right ) \le Cn^{-{\delta_v }/{2} }(1+|x|)^{-\lambda } , 
\end{equation}
where $\delta_v=\min\{a-2,v\}$ and $\lambda = \min \left \{ r-v,a-\beta v,a(1-\frac{v}{r} ) \right \}$.
\end{lem}
\begin{proof}
We just prove \eqref{s406}. The proof of \eqref{s407} is similar. The proof is similar to that of Lemma \ref{s4l03}. Similarly to \eqref{s404}, we have 
\begin{eqnarray}\label{s408}
&\quad&\mathbb P\left(\frac{\log W_m }{\sqrt{n }\sigma } +\frac{S_n}{\sqrt{n }\sigma } \ge x-\alpha_n, \frac{S_n}{\sqrt{n }\sigma }\le x  \right)\nonumber\\
&=&\int \mathbf 1_{\left\{\alpha_n-z \ge0\right\}}\left[F_{n-m}\left(\mathcal A\right)-F_{n-m}\left(\mathcal B\right)\right]\nu_n(\mathrm{d}y,\mathrm{d}z),
\end{eqnarray}
where $\mathcal A=\frac{\sqrt{n} }{\sqrt{n-m} }(x-y) $ and $\mathcal B=\frac{\sqrt{n} }{\sqrt{n-m} }(x-\alpha_n-y-z)$.
By Lemma \ref{s4l1}(a), for any $x \in \mathbb R$,
\begin{equation}\label{s409}
\left|F_n(x)-\Phi (x)\right|\le Cn^{-{\delta }/{2} } (1+|x|)^{-a} . 
\end{equation}
By the mean value theorem,  we have
\begin{eqnarray}\label{s409+}
&\quad&\left|\Phi\left(\mathcal A\right)- \Phi \left(\mathcal B\right)\right|\nonumber\\
&\le& \varphi (\zeta )\frac{\sqrt{n} }{\sqrt{n-m} } |\alpha_n+z|\nonumber\\
&\le&C\varphi (\zeta )|\alpha_n+z|,
\end{eqnarray}
where $\zeta$ takes value between $\mathcal A$ and $\mathcal B$. Notice that 
\begin{equation}\label{s411}
\varphi (\zeta)\le\max\left \{ \varphi \left(\mathcal A\right),\varphi \left(\mathcal B\right) \right \} 
\end{equation}
whenever $(x-y)(x-y-\alpha_n-z)\ge 0$, so that
combining \eqref{s409+} and \eqref{s411} gives 
\begin{eqnarray}\label{s410}
\left|\Phi \left(\mathcal A\right)-\Phi  \left(\mathcal B \right)\right|
\le C\min\left\{1,\;\left(e^{-\frac{1}{2}(x-y)^{2}  }+e^{-\frac{1}{2}(x-y-\alpha_n-z)^{2}  } \right)|\alpha_n+z|\right\}.
\end{eqnarray}
Moreover,
notice that 
\begin{equation}\label{eset}
\{(x-y)(x-y-\alpha_n-z)<0\}\subset\{|y|>\frac{1}{2}|x|\}\cup \{|\alpha_n+z|>\frac{1}{2}|x|\}.
\end{equation}
Thus,
combining \eqref{s408} with \eqref{s409} and \eqref{s410}, and using \eqref{eset},  we see that
\begin{eqnarray}\label{s413}
&\quad&\mathbb P\left(\frac{\log W_m }{\sqrt{n }\sigma } +\frac{S_n}{\sqrt{n }\sigma } \ge x-\alpha_n, \frac{S_n}{\sqrt{n }\sigma }\le x  \right)\nonumber\\
&\le& C(I_{n1}+I_{n2}+I_{n3}+I_{n4}+I_{n5}+I_{n6}),
\end{eqnarray}
where
\begin{eqnarray*}
I_{n1}&=&n^{-{\delta }/{2} } \int (1+|x-y|)^{-a}   \nu_n(\mathrm{d}y,\mathrm{d}z),\\
I_{n2}&=&n^{-{\delta }/{2} } \int (1+|x-\alpha_n-y-z|)^{-a}  \nu_n(\mathrm{d}y,\mathrm{d}z),\\
I_{n3}&=&\int e^{-\frac{v}{2}(x-y)^{2}  } |\alpha_n+z|^v \nu_n(\mathrm{d}y,\mathrm{d}z),\\
I_{n4}&=&\int e^{-\frac{v}{2}(x-\alpha_n-y-z)^{2}  }  |\alpha_n+z|^v\nu_n(\mathrm{d}y,\mathrm{d}z),\\
I_{n5}&=&\int_{|y|>\frac{1}{2}|x|}\nu_n(\mathrm{d}y,\mathrm{d}z),\\
I_{n6}&=&\int_{|\alpha_n+z|>\frac{1}{2}|x|}\nu_n(\mathrm{d}y,\mathrm{d}z).
\end{eqnarray*}

We first deal with $I_{n5}$ and $I_{n6}$. Let $c$ be a positive constant. We can calculate that by \eqref{s409},
\begin{eqnarray}\label{s414}
\int _{|y|>c|x|}\nu_n(\mathrm{d}y,\mathrm{d}z )
&=&\mathbb P\left(\frac{|S_m|}{\sqrt{n}   }>c|x| \right)\nonumber\\
&=&1-F_m\left(c\frac{\sqrt{n} }{\sqrt{m} } |x| \right)+F_m\left(-c\frac{\sqrt{n} }{\sqrt{m} } |x| \right)\nonumber\\
&\le& 2\left[1-\Phi \left(c\frac{\sqrt{n} }{\sqrt{m} } |x| \right)\right]+\left|F_m\left(c\frac{\sqrt{n} }{\sqrt{m} } |x| \right)-\Phi \left(c\frac{\sqrt{n} }{\sqrt{m} } |x| \right)\right|\nonumber\\
&\quad&+\left|F_m\left(-c\frac{\sqrt{n} }{\sqrt{m} } |x| \right)-\Phi \left(-c\frac{\sqrt{n} }{\sqrt{m} } |x| \right)\right|\nonumber\\
&\le& 2\left[1-\Phi \left(c\frac{\sqrt{n} }{\sqrt{m} } |x| \right)\right] + 2m^{-{\delta }/{2} } \left(1+c\frac{\sqrt{n} }{\sqrt{m} } |x|\right)^{-a} \nonumber\\
&\le& Cn^{-\delta/2}(1+|x|)^{-a }.  
\end{eqnarray}
Thus 
\begin{equation}\label{ein5}
I_{n5}\leq C n^{-\delta/2}(1+|x|)^{-a }.
\end{equation}
Since $\sup_n\mathbb E\left|\log W_n\right|^r<\infty$ for some $r>0$, we have
\begin{eqnarray}\label{s415}
\int _{|z|>c|x|}\nu_n(\mathrm{d}y, \mathrm{d}z)
&=&\mathbb P\left(\left|\frac{\log W_m}{\sqrt{n}  } \right|> c|x|\right)\nonumber\\
&\le& C\sup_n\mathbb E\left|\log W_n\right|^{r} n^{-{r}/{2} } |x|^{-r} \nonumber\\
&\le& C n^{-r/2} (1+|x|)^{-r}.
\end{eqnarray}
On the set $ \{|\alpha_n+z|>\frac{1}{4}|x|\}$, we have $|z|>\frac{1}{4}|x|-\alpha_n\ge \kappa |x|$ for some constant $\kappa>0$. From \eqref{s415} we derive
\begin{equation}\label{ein6}
I_{n6}\le\int_{|\alpha_n+z|>\frac{1}{4} |x|}\nu_n(\mathrm{d}y, \mathrm{d}z)\le \int _{|z|>\kappa |x|}\nu_n(\mathrm{d}y, \mathrm{d}z)\leq C n^{-r/2}(1+|x|)^{-r}.
\end{equation}

Next we consider $I_{n1}$ and $I_{n2}$.
For $I_{n1}$, by \eqref{s414}, 
\begin{eqnarray}\label{s416}
I_{n1} 
&\le& n^{-{\delta }/{2} }  \left[\int _{|y|\le\frac{1}{2}|x| }(1+|x-y|)^{-a}\nu_n(\mathrm{d}y,\mathrm{d}z) +\int _{|y|>\frac{1}{2}|x| }\nu_n(\mathrm{d}y,\mathrm{d}z)\right]\nonumber\\
&\le& Cn^{-{\delta }/{2} } \left[\left(1+\frac{1}{2}|x| \right)^{-a} +(1+|x|)^{-a}\right]\nonumber\\
&\le& Cn^{-{\delta }/{2} }(1+|x|)^{-a}.
\end{eqnarray}
Similarly, for $I_{n2}$, by \eqref{s414} and \eqref{ein6},
\begin{eqnarray}\label{s417}
I_{n2} 
&\le& n^{-{\delta }/{2} }  \left[\int _{|y+z+\alpha_n|\le\frac{1}{2}|x| }\left(1+|x-\alpha_n-y-z|\right)^{-a }\nu_n(\mathrm{d}y,\mathrm{d}z) \right.\notag\\
&&\left. +\int _{|y|>\frac{1}{4}|x| }\nu_n(\mathrm{d}y,\mathrm{d}z)+\int _{|\alpha_n+z|>\frac{1}{4}|x| }\nu_n(\mathrm{d}y,\mathrm{d}z)\right]\nonumber\\
&\le& Cn^{-{\delta }/{2} } \left[\left(1+\frac{1}{2}|x| \right)^{-a } +(1+|x|)^{-a }+\left(1+|x|\right)^{-r  }\right]\nonumber\\
&\le& Cn^{-{\delta }/{2} } \left(1+|x|\right)^{-\lambda _1} ,
\end{eqnarray}
where $\lambda_1 =\min \left \{ r,a \right \}$.

Finally, we deal with $I_{n3}$ and $I_{n4}$. For any constant $c>0$, using \eqref{s414}, we deduce that
\begin{eqnarray}\label{s418}
&\quad&\int e^{-c(x-y)^2}\nu_n(\mathrm{d}y,\mathrm{d}z) \nonumber\\
&\le& \int _{|y| \le \frac{1}{2} |x|}e^{-c(x-y)^2} \nu_n(\mathrm{d}y,\mathrm{d}z)+\int _{|y| > \frac{1}{2} |x|} \nu_n(\mathrm{d}y,\mathrm{d}z)\nonumber\\
&\le& e^{-\frac{c}{4}x^2 } + C(1+|x|)^{-a}.
\end{eqnarray}
It follows  by Hölder inequality that  for $s=\frac{r}{r-v}$, 
\begin{eqnarray}\label{s419}
&\quad&\int e^{-c(x-y)^2}|z|^v\nu_n(\mathrm{d}y,\mathrm{d}z) \nonumber\\
&\le& \left(\int e^{-cs(x-y)^2} \nu_n(\mathrm{d}y,\mathrm{d}z) \right)^{{1}/{s} } \left(\int |z|^r\nu_n(\mathrm{d}y,\mathrm{d}z) \right)^{{v}/{r} }\nonumber\\
&\le& \left[e^{-\frac{cs}{4}x^2} + C(1+|x|)^{-a} \right]^{{1}/{s} }\left(\mathbb E\left|\frac{\log W_m}{\sqrt{n}\sigma  } \right|^r\right)^{{v}/{r} }\nonumber\\
&\le& Cn^{-{v}/{2} }\left[e^{-\frac{cs}{4}x^2 } + C(1+|x|)^{-a} \right]^{{1}/{s} }\nonumber\\
&\le& Cn^{-{v}/{2} }(1+|x|)^{-a/{s} } .
\end{eqnarray}
Thus, by \eqref{s418} and \eqref{s419},
\begin{eqnarray}\label{s420}
I_{n3} 
&\le& C\alpha_n^v\int e^{-\frac{v}{2} (x-y)^2}\nu_n(\mathrm{d}y,\mathrm{d}z) +\int e^{-\frac{v}{2} (x-y)^2}|z|^v\nu_n(\mathrm{d}y,\mathrm{d}z)\nonumber\\
&\le& C\left[n^{-{\delta }/{2} } (1+|x|)^{-(a-\beta)}+n^{-{v}/{2} } (1+|x|)^{-a(1-\frac{v}{r} )} \right] \nonumber\\
&\le& Cn^{-{\delta_v }/{2} }(1+|x|)^{-\lambda} .
\end{eqnarray}
For $I_{n4}$, similarly, we can obtain
\begin{eqnarray}\label{s421}
I_{n4}
&\le& C\left[n^{-{\delta }/{2} } (1+|x|)^{-(\lambda_1-\beta) }+n^{-{v}/{2} } (1+|x|)^{-\lambda_1(1-\frac{v}{r} )} \right] \nonumber\\
&\le& Cn^{-{\delta_v }/{2} }(1+|x|)^{-\lambda} .
\end{eqnarray}
Combing \eqref{s413} with \eqref{ein5}, \eqref{ein6}, \eqref{s416}, \eqref{s417}, \eqref{s420} and \eqref{s421} leads to \eqref{s406}.
\end{proof}

\section{Proofs of theorems}\label{BS5}
In this section, we shall gives   proofs of the main theorems associated to the convergence rates in the central limit theorem on $\log Z_n$ which are introduced in Section \ref{BS1}.

\begin{proof}[Proof of Theorem \ref{tt2}]
By  Theorem \ref{lm}, we have $\sup_n\mathbb E\left| \log W_n  \right|^{r_1}<\infty$ for some $r_1\in(r,a)$ (we set $r=1$ if the condition (i) holds). 
Set $m=[\sqrt n]$ and $a_n=n^{-{\min\{a-2,1\}}/{(2v)}}$, where  $v=\min\{r,1\}$.
Observe that 
\begin{eqnarray*}
	&&\mathbb P\left(\frac{\log Z_n -n\mu}{\sqrt{n}\sigma }\leq x\right)\\
	&=&F_n(x)+\mathbb P\left(\frac{\log Z_n -n\mu}{\sqrt{n}\sigma }\leq x,\frac{S_n}{\sqrt n\sigma }>x \right)-\mathbb P\left(\frac{\log Z_n -n\mu}{\sqrt{n}\sigma }> x,\frac{S_n}{\sqrt n\sigma }\le x\right).
\end{eqnarray*}
According to Lemma \ref{s4l1}(a), we have \eqref{s403}. It remains  to prove 
\begin{equation}\label{s501}
	\sup_{x\in \mathbb R}\mathbb P\left(\frac{\log Z_n -n\mu}{\sqrt{n}\sigma }\ge x,\frac{S_n}{\sqrt n\sigma }\le x\right)\le Cn^{-{\delta }/{2}},
\end{equation}
\begin{equation}\label{s502}
	\sup_{x\in \mathbb R}\mathbb P\left(\frac{\log Z_n -n\mu}{\sqrt{n}\sigma }\le x,\frac{S_n}{\sqrt n\sigma }\ge x\right)\le Cn^{-{\delta }/{2}}  .
\end{equation}
Then we only prove \eqref{s501}, and the proof of \eqref{s502} is similar.  For $m<n$, denote
\begin{equation*}
	D_{m,n}=\frac{\log W_n-\log W_m}{\sqrt n \sigma }.
\end{equation*}
Then 
\begin{eqnarray}\label{s503}
	&&\mathbb P\left(\frac{\log Z_n -n\mu}{\sqrt{n}\sigma }\ge x,\frac{S_n}{\sqrt n\sigma }\le x\right)
	\notag\\
	&\le &
	\mathbb P\left(\frac{\log W_m}{\sqrt{n}\sigma }+\frac{S_n}{\sqrt n\sigma }\ge x-a_n,\frac{S_n}{\sqrt n\sigma }\le x\right)+\mathbb P\left ( \left | D_{m,n} \right |>a_n  \right ) .
\end{eqnarray}
 Without loss of generality, we  can  think that $0<\alpha<\min\{1, p\varepsilon, 2\varepsilon\}$. By Proposition  \ref{s3l02},
\begin{eqnarray*}
	\mathbb P\left ( \left | D_{m,n} \right | >a_n \right ) 
	&=&\mathbb P\left ( \left | \log W_n-\log W_m \right | >\sqrt n a_n\sigma\right )\nonumber\\
	&\le& C\left [ \left ( \sqrt{n}a_n\sigma   \right ) ^{-\alpha } \rho ^{m}  +m^{-r} \right ] \nonumber\\
	&\le& Cn^{-{\delta}/{2} }.
\end{eqnarray*}
According to Lemma \ref{s4l03}, we   know that the first term in \eqref{s503} is bounded by $Cn^{-\min\{{a-2,v }\}/{2} }$, which completes the proof.
\end{proof}
	
\begin{lem}\label{5llw}
Assume that $\mathbb E\left(\log m_0(\mathbf t)\right)^{a} < \infty$ for some $a>2$ and $\sup_n \mathbb E \left|\log W_n\right|^r < \infty$ for some $r>0$.  Let $v\in(0,1]$ be a constant satisfying $v<r$. Set $\delta=\min\{a-2,1\}$, $m=[\sqrt{n}]$ and $\alpha_n=n^{-\delta/(2v)}|x|^\beta$, with $\beta\in(0,1)$. Put $\delta_v=\min\{a-2,v\}$. If
\begin{equation}\label{5lw}
\mathbb E|\log W_n-\log W_m|^q=O(m^{-\gamma})
\end{equation}
for some $q>0$ and $\gamma>0$ satisfying $\gamma+q(1-\delta/v)\geq \delta_v$, then for  $|x|\geq 1$,
\begin{equation}\label{5e1}
\mathbb P\left(\frac{\log Z_n -n\mu}{\sqrt{n}\sigma }\ge x,\frac{S_n}{\sqrt n\sigma }\le x\right)\le C n^{-\delta_v/2}(1+|x|)^{-\lambda},
\end{equation}
\begin{equation}\label{5e2}
\mathbb P\left(\frac{\log Z_n -n\mu}{\sqrt{n}\sigma }\le x,\frac{S_n}{\sqrt n\sigma }\ge x\right)\le C n^{-\delta_v/2}(1+|x|)^{-\lambda},
\end{equation}
where $\lambda=\min\{\beta q, r-v,a-\beta v, a(1-\frac{v}{r})\}$.
\end{lem}
\begin{proof}
We just prove \eqref{5e1}, and \eqref{5e2} can be proved similarly.  It can be seen that  \eqref{s503}  still holds with $\alpha_n$ in place of $a_n$. By \eqref{5lw}, 
\begin{eqnarray}\label{5e3}
\mathbb P\left ( \left | D_{m,n} \right | >\alpha_n \right ) 
	&=&\mathbb P\left ( \left | \log W_n-\log W_m \right | >n^{(v-\delta)/(2v)}\sigma|x|^\beta\right )\nonumber\\
	&\le& C\frac{\mathbb E|\log W_n-\log W_m|^q}{n^{q(v-\delta)/(2v)}|x|^{\beta q}}\nonumber\\
&\le& C n^{-\frac{\gamma}{2}-\frac{q(v-\delta)}{2v}}	(1+|x|)^{-\lambda}\nonumber\\
 &\le&C n^{-\delta_v/2}(1+|x|)^{-\lambda}.
\end{eqnarray}
Moreover, we see that    \eqref{s406} holds from Lemma \ref{sl403}. The proof will be finished  by  combining \eqref{s503} (with $\alpha_n$ in place of $a_n$) with \eqref{5e3} and \eqref{s406}.
\end{proof}

\begin{proof}[Proof of Theorem \ref{tt20}]
For $|x|\leq1$, noticing that $(1+|x|)^{-\lambda}\ge2^{-\lambda}$, by Theorem \ref{tt2}, we deduce that
\begin{equation*} 
\left | \mathbb{P}\left ( \frac{ \log_{}{Z _{n}  }  -n\mu   }{\sqrt{n} \sigma  }\le x \right ) -\Phi \left ( x \right )  \right |\le Cn^{-{\delta }/{2} }\le Cn^{-{\delta }/{2}}(1+|x|)^{-\lambda}.
	\end{equation*}	
	
Now we consider $|x|\ge 1$. Similarly to the proof of Theorem \ref{tt2}, it suffices to show that   \eqref{5e1} and \eqref{5e2} hold with $\delta_v=\delta=\min\{a-2,1\}$.    Set $m=[\sqrt{n}]$ and $\alpha_n=n^{-\delta/2}|x|^\beta$, with $\beta\in(0,1)$. We shall use  Lemma \ref{5llw}.

We first  work on the assertion (a). Take $\lambda\in(0, a)$. Under the conditions of the assertion (a), by Theorem \ref{lm}, we have $\sup_n \mathbb E|\log W_n|^{r^2}<\infty$ for all $r>1$. By Proposition   \ref{s3l03}, 
$$\mathbb E|\log W_n-\log W_m|^r=O(m^{-(r-1)}).$$
Applying Lemma \ref{5llw} to $q=r$, $v=1$ and $\gamma=r-1$, and taking $\beta\in(0,1)$ small enough and  $r>1$ large enough such that 
\begin{equation*}
\begin{cases}
r-1+r(1-\delta)\geq \delta,\\
\lambda<\min\left\{\beta r, r^2-1, a-\beta, a(1-\frac{1}{r^2})\right\},
\end{cases}
\end{equation*}
we see that  \eqref{5e1} and \eqref{5e2} hold.

Next we deal with the assertion (b). By Theorem \ref{lm}, there exists $r$ satisfying $\max\{2\delta,1\}\le r<a$ such that $\sup_n \mathbb E|\log W_n|^r<\infty$. By Proposition   \ref{s3l03}, 
$$\mathbb E|\log W_n-\log W_m|=O(m^{-(r-1)}).$$
Noticing that $r\ge 2\delta$, and applying Lemma \ref{5llw} to $q=1$, $v=1$ and $\gamma=r-1$, we see that  \eqref{5e1} and \eqref{5e2} hold for $\lambda <\min\left\{\beta , r-1,a-\beta, a(1-\frac{1}{r})\right\}$.
\end{proof}

\begin{proof}[Proof of Corollary \ref{cobs}]
Set  $m=[\sqrt{n}]$ and $\alpha_n=n^{-\min\{a-2,1\}/(2v)}|x|^\beta$, where $\beta\in(0,1)$   and $v>0$   will be determined later.  The conditions of Corollary \ref{cobs} ensure $\sup_n\mathbb E |\log W_n|^r<\infty$ for $2<r<a^*$, and for any $\gamma>0$,
\begin{equation*} 
\mathbb E|\log W_n-\log W_m|^r=O(m^{-\gamma}).
\end{equation*}

For $\lambda<\lambda_0$, take $v=1$, $\delta=\min\{a-2,1\}$, and  choose $\beta\in(0,1)$ and $0<r<a<a^*$ large enough such that 
$$\lambda <\min\left\{\beta r, r-1, a-\beta, a(1-\frac{1}{r})\right\}.$$
Then take $\gamma>0$  such that $\gamma +r(1-\delta)\ge\delta$.
It follows from Lemma \ref{5llw} that \eqref{5e1} and \eqref{5e2} hold for this $\lambda$ and $\delta_v=\min\{a-2,1\}=\delta$. According to the proof of Theorem \ref{tt20}, we can obtain the desired result.

Now we consider the case that $a^*<\infty$ and $\lambda\in[\lambda_0, a^*)$. Fix $0<\delta<\min\{a^*-2, a^*-\lambda\}$. Choose $\beta\in(0,1)$ and  $\lambda <r<a<a^*$ large enough such that 
\begin{eqnarray*}
\delta< \min\{a-2,v\} \quad\text{and}\quad
\lambda=\min\left\{\beta r, r-v, a-\beta v, a(1-\frac{v}{r})\right\},
\end{eqnarray*}
where $v=r-\lambda$. Set $\delta_v=\min\{a-2,v\}$. Since we can take $\gamma>0$ large enough such that $\gamma +r(1-\min\{a-2,1\}/v)\ge\delta_v$,  by Lemma \ref{5llw}, we derive that \eqref{5e1} and \eqref{5e2} hold for this $\lambda$ and $\delta_v=\min\{a-2,v\}>\delta$.
\end{proof}

	\begin{proof}[Proof of Theorem \ref{tt3}]
	Take $1<r<a$ satisfying $p > \max\{1+\frac{a}{(a-1)\varepsilon},\frac{2a}{(a-1)\varepsilon}\}$.
	By Theorem \ref{lm}, we have $\sup_n \mathbb E|\log W_n|^r<\infty$.
		Set $m=[n^\beta]$ and $a_n=n^{-b}$, where $\beta\in(\frac{1}{2r}, \frac{1}{2})$ and $b>1/2$. Notice that 
		\begin{equation}\label{eo}
			\mathbb P\left(\frac{\log Z_n -n\mu}{\sqrt{n}\sigma }\leq x\right)\left\{
			\begin{array}{l}
				\leq \mathbb P\left(\frac{\log Z_{m} -m\mu}{\sqrt{n}\sigma }+\frac{S_n-S_{m} }{\sqrt{n}\sigma }\leq x+a_n\right)+\mathbb P(|D_{m,n}|>a_n)\\
				\geq \mathbb P\left(\frac{\log Z_{m} -m\mu}{\sqrt{n}\sigma }+\frac{S_n-S_{m}}{\sqrt{n}\sigma }\leq x-a_n\right)-\mathbb P(|D_{m,n}|>a_n)
			\end{array}.
			\right.
		\end{equation}
		It suffices to show that as $n$ tends to infinity, 
		\begin{equation}\label{t4ee1}
			\sqrt{n}\mathbb P(|D_{m,n}|>a_n)\rightarrow 0
		\end{equation}
		and
		\begin{eqnarray}\label{t4ee2}
			&&\mathbb P\left(\frac{\log Z_{m} -m\mu}{\sqrt{n}\sigma }+\frac{S_n-S_{m} }{\sqrt{n}\sigma }\leq x\pm a_n\right)\nonumber\\&=&\Phi(x)-\frac{1}{\sigma\sqrt{n}}\varphi(x)\mathbb E\log W+\frac{1}{\sqrt{n}}Q(x)+\frac{1}{\sqrt{n}}o(1).
		\end{eqnarray}
We can prove \eqref{t4ee2} by a way similar to \cite{HZG}, so here we omit its proof and just prove
 \eqref{t4ee1}. We   think that $0<\alpha<\min\{1, p\varepsilon, 2\varepsilon\}$.
 By Proposition  \ref{s3l02},
		\begin{eqnarray*}
	\sqrt{n}\mathbb P\left ( \left | D_{m,n} \right | >a_n \right ) 
	&=&\sqrt{n}\mathbb P\left ( \left | \log W_n-\log W_m \right | >\sqrt n a_n\sigma\right )\nonumber\\
	&\le& C\sqrt{n}\left [ \left ( \sqrt{n}a_n\sigma   \right ) ^{-\alpha } \rho ^{n^\beta}  +n^{-\beta r} \right ] \nonumber\\
	&\le& Cn^{1/2-\beta r}\to 0 
\end{eqnarray*}
as $n$ tends to infinity, which gives  \eqref{t4ee1}.
	\end{proof}

\end{document}